\documentclass[12pt,a4paper]{amsart}
\usepackage[utf8]{inputenc}
\usepackage[T1]{fontenc}
\usepackage[english]{babel}
\usepackage[all]{xy}
\usepackage{fancyhdr}
\usepackage{array}
\usepackage{amsfonts}
\usepackage{bbm}
\usepackage{amsmath}
\usepackage{amssymb}
\usepackage{graphicx}
\usepackage{stmaryrd}
\usepackage{moreverb}
\usepackage{geometry}
\usepackage{amsthm}
\usepackage{color}
\usepackage{bbold}
\usepackage[style = alphabetic, maxbibnames = 99, backend = bibtex]{biblatex}
\usepackage{yhmath}
\usepackage{tikz}
\usepackage{hyperref}
\usepackage{cleveref}
\usepackage{csquotes}

\title[Stability of adapted tangent sheaves on KE log Fano pairs]{Stability properties of adapted tangent sheaves on K\"ahler--Einstein log Fano pairs}
\author{Louis Dailly}
\address{Louis Dailly, Institut de Math\'ematiques de Toulouse, Universit\'e Paul Sabatier, 31062 Toulouse Cedex~9, France}
\email{\href{mailto:louis.dailly@utoulouse.fr}{louis.dailly@utoulouse.fr}}

\newcommand{\incl}[1][r]
  {\ar@<-0.2pc>@{^(-}[#1] \ar@<+0.2pc>@{-}[#1]}
\newcommand{\immouv}[1][r]
   {\ar@{}[#1] |*[o][F]{\hbox{%
         \vrule width 1.5mm height 0pt depth 0pt%
         \vrule width 0pt height .75mm depth .75mm%
         }}
     \ar@{^{(}->}[#1]}

\newcommand{\Q}{\mathbb{Q}}
\newcommand{\R}{\mathbb{R}}
\newcommand{\C}{\mathbb{C}}

\newcommand{\PP}{\mathbb{P}}

\newcommand{\reg}{\mathrm{reg}\,}
\newcommand{\orb}{\mathrm{orb}\,}

\newcommand{\End}{\mathrm{End}\,}

\newcommand{\PGL}{\mathrm{PGL}\,}
\newcommand{\id}{\mathrm{id}}
\newcommand{\inv}{\mathrm{inv}}

\newcommand{\ddc}{\mathrm{dd^c}}
\newcommand{\Ric}{\mathrm{Ric}\,}
\newcommand{\tr}{\mathrm{tr}\,}

\theoremstyle{plain}
\newtheorem{theo}{Theorem}[section]
\newtheorem*{theoA}{Theorem A}
\newtheorem*{coroB}{Corollary B}
\newtheorem{lemm}[theo]{Lemma}
\newtheorem{prop}[theo]{Proposition}

\theoremstyle{definition}
\newtheorem{defi}[theo]{Definition}

\newtheorem{rema}[theo]{Remark}
\newtheorem{exem}[theo]{Example}
\newtheorem{nota}[theo]{Notation}
\newtheorem*{ackn}{Acknowledgement}

\newenvironment{itemize*}
    {\begin{itemize}%
      \setlength{\itemsep}{0pt}%
      \setlength{\parskip}{0pt}}%
    {\end{itemize}}

\bibliography{mabiblio}
\AtBeginBibliography{\tiny}

\begin{document}

\begin{abstract}
Let $(X, \Delta)$ be a log Fano pair with standard coefficients endowed with a singular K\"ahler--Einstein metric. We show that the adapted tangent sheaf $\mathcal{T}_{X, \Delta, f}$ and the adapted canonical extension $\mathcal{E}_{X, \Delta, f}$ are polystable with respect to $f^*c_1(X, \Delta)$ for any strictly $\Delta$-adapted morphism $f : Y \to X$.
\end{abstract}

\maketitle

\tableofcontents

\section{Introduction}

Let $(X, \omega)$ be a K\"ahler compact manifold. The Kobayashi--Hitchin correspondence establishes that a holomorphic vector bundle $E$ over $X$ can be endowed with a Hermite--Einstein metric $h$ with respect to $\omega$ if and only if it is polystable with respect to its associated de Rham class $\left[ \omega \right]$. This type of result is an algebro-geometric correspondence, which means that there is an equivalence between differential geometric notions and algebraic ones.
Let us assume that $X$ is a Fano manifold endowed with a K\"ahler--Einstein metric $\omega$, that means:
\begin{itemize}
	\item the first Chern class of $X$, denoted $c_1(X)$ is a K\"ahler class,
	\item the metric $\omega$ satisfies the following equality of $(1,1)$-forms: $\Ric \omega = \omega$.
\end{itemize}

Then, the metric $\omega$ is Hermite--Einstein with respect to $\omega$ itself, so that $T_X$ is polystable. The following result due to Tian expresses that $T_X$ satisfies a stronger notion of polystability, namely the canonical extension $E_X$ is also polystable:

\begin{theo}\textup{\cite[Thm.~0.1]{Tia92}}
	Let $X$ be a Fano manifold endowed with a K\"ahler--Einstein metric $\omega$. Then, the canonical extension $E_X$ is polystable with respect to $c_1(X)$.	
\end{theo}

Several generalizations of this result to the singular setting exist: 
\begin{itemize}
	\item Li proved the semistability of the orbifold tangent bundle and the canonical extension when $(X, \Delta)$ is a log smooth log Fano pair \cite{Li21}
	\item Druel, Guenancia and P\u{a}un showed the polystability of the tangent sheaf and the canonical extension when $X$ is a $\Q$-Fano variety \cite{DGP24}.
\end{itemize}

The aim of the present paper is to study the case where $(X, \Delta)$ is a general log Fano pair and $\Delta$ has standard coeffcients. In this context, we have to deal with orbisheaves which take into account the presence of a boundary divisor and ramification phenomena through orbifold structure (see \cite{GT22}). These orbisheaves can be realized as concrete sheaves through \emph{adapted morphisms}. They are morphisms $f :Y \to X$ that are finite, Galois, and they branch over irreducible components $\Delta_i$ with ramification indices prescribed by their coefficients in $\Delta$. In order to deal with concrete sheaves, a general notion of adapted sheaves has been introduced in \cite{Dai25} that generalizes the notion of \emph{sheaf of adapted differentials} $\Omega^{[1]}_{X, \Delta, f}$ and of \emph{adapted tangent sheaf} $\mathcal{T}_{X, \Delta, f}$ introduced by Miyaoka, Campana (see \cite{CKT16}). Hence, the notion of \emph{adapted canonical extension} $\mathcal{E}_{X, \Delta, f}$ is defined.
The definition of singular K\"ahler--Einstein metrics in this setting is given in \cite{BBEGZ19}.

With these notations, we establish the following result that is an extension of Tian's strong semistability result.

\begin{theoA}\label{Main Theorem}
    Let $(X, \Delta)$ be a log Fano pair endowed with a singular K\"ahler--Einstein metric. Then, for any strictly adapted morphism $f \colon Y \to X$, the adapted tangent sheaf $\mathcal{T}_{X, \Delta, f}$ and the adapted canonical extension $\mathcal{E}_{X, \Delta, f}$ are polystable with respect to $f^*c_1(X, \Delta)$.
\end{theoA}

Combined with the main result from \cite{Dai25}, we get the following uniformization result of log Fano pairs satisfying the equality case in the Miyaoka--Yau equality.

\begin{coroB}
	Let $(X, \Delta)$ be a log Fano pair of dimension $n$. Assume that there exists a singular K\"ahler--Einstein metric $\omega$ on $(X, \Delta)$ and the following equality of Chern classes holds:
	$$
	\left( 2(n+1)c_2(X, \Delta) - nc_1(X, \Delta)^2 \right) \cdot c_1(X, \Delta)^{n-2} = 0.
	$$
	Then, there exists a finite group $G \subset \PGL(n+1, \C)$ such that $(X, \Delta) \simeq (\PP^n / G, \Delta_G)$.
\end{coroB}

\emph{Strategy of the proof}\\
	The proof is based on the strategy developed in \cite{DGP24}, and was already used in \cite{Eno88, Gue16, GT22}. The problem of stability properties can be reduced to a log resolution of the pair $(X, \Delta)$. By considering the strict transform $\widehat{\Delta}$ of $\Delta$, the pair $(\widehat{X}, \widehat{\Delta})$ is log smooth, it can be endowed with a smooth orbi-\'etale orbistructure in the sense of \cite{GT22}. Now, the pullback of the K\"ahler--Einstein metric on $\widehat{X}$ can be approximated by orbifold metrics that are solutions of perturbed Monge--Ampère equations, and semistability properties of orbifold bundles on the log resolution can be derived from this approximation process. The strategy developed by Druel, Guenancia, P\u{a}un can be adapted in this orbifold setting.
	
	In order to get semistability of the adapted sheaves, we construct an adapted cover over the resolution $(\widehat{X}, \widehat{\Delta})$ such that adapted sheaves are semistable in a similar way as \cite{GT22}. To conclude, we show that the existence of this adapted cover is sufficient to deduce the semistability properties of adapted sheaves on adapted covers over $X$. The polystability of the adapted sheaves is obtained similarly to \cite{DGP24}. 
	
\begin{ackn} 
    The author warmly thanks Benoît Claudon and Henri Guenancia for their support through discussions, suggestions during his research and corrections of the present paper.
\end{ackn}

\section{Setup}

\subsection{Singular K\"ahler--Einstein metric and resolutions}

\begin{defi}
    \emph{A log Fano pair} $(X, \Delta)$ is the datum of a projective normal variety $X$, and a $\Q$-Weil divisor $\Delta$ with standard coefficients, such that $(X, \Delta)$ is klt and the divisor $-(K_X + \Delta)$ is $\Q$-ample.
\end{defi}

There exists a positive integer $m$ such that $m(K_X + \Delta)$ is Cartier.
Hence, there exists a smooth metric $h^m$ on $m(K_X + \Delta)$ such that the smooth form:
$$
\omega := \frac{1}{m} \textbf{i}\Theta(-m(K_X + \Delta), {h^m}^\vee) \in c_1(X, \Delta)
$$
is K\"ahler. On the other hand, for each local trivialization $\tau_\alpha$ of $m(K_X + \Delta)$, we can define the local $(n,n)$-form $\displaystyle \left(\textbf{i}^{mn^2}\frac{\tau_\alpha \wedge \overline{\tau_\alpha}}{|\tau_\alpha|_{h^m}^2} \right)^{\frac{1}{m}}$. These local forms glue together to define a global form on $X^* := X_\reg \backslash |\Delta|$, and we define the adapted volume $\mu_{h^m}$ as the measure induced by the integration against this $(n,n)$-form on $X^*$.

\begin{defi}
    A \emph{K\"ahler--Einstein metric} is the datum of an $\omega$-psh function $\Phi \in L^\infty(X) \cap \mathrm{PSH}(X, \omega)$ on $X$ satisfying the following Monge--Ampère equation:
    \begin{equation}\label{KE}
    \left( \omega + \ddc \Phi \right)^n = e^{-\Phi} \mu_{h^m}.
    \end{equation}
\end{defi}

\begin{nota}
    Let us consider a log-resolution $\pi : \widehat{X} \longrightarrow X$ for the pair $(X, \Delta)$. We have:
    $$
    K_{\widehat{X}} + \widehat{\Delta} = \pi^*(K_X + \Delta) + \sum_{1 \leq i \leq k} a_i E_i,
    $$
    where $\widehat{\Delta} := \pi^{-1}_* \Delta$ is the strict transform of $\Delta$, $E_i$ are the exceptional divisors and $a_i > -1$. In the following, we will consider $E = \sum_i E_i$ the reduced exceptional divisor. 
\end{nota}

The pair $(\widehat{X}, \widehat{\Delta})$ is log smooth, hence it is naturally endowed with an orbifold structure. In order to deal with orbifold objects, let us recall some related definitions (see \cite[Sect.~2]{Dai25} for details).

\begin{defi}\label{Def orbistructure orbi-etale}
    Let $(X, \Delta)$ be a pair. An \emph{orbi-\'etale orbistructure} on $(X, \Delta)$ is a collection $\mathcal{C}$ of triples $(U_\alpha, V_\alpha, \eta_\alpha)$ called local uniformizing charts (l.u.c.) such that:
    \begin{itemize}
        \item $(U_\alpha)_\alpha$ is an open cover of $X$,
        \item $V_\alpha$ is a normal complex space,
        \item $\eta_\alpha : V _\alpha \to U_\alpha$ is a surjective finite Galois morphism, such that:
        \begin{itemize}
            \item $\eta_\alpha^*\Delta_i = m_i \Delta'_i$, where $\Delta'_i$ is reduced, and $\eta_\alpha$ is \'etale above $U_{\alpha, \reg} \backslash |\Delta_{|U_\alpha}|$,
            \item Let us define the \emph{intersection of $V_\alpha$ with $V_\beta$}, denoted $V_{\alpha\beta}$, as the normalization of $V_\alpha \underset{X}{\times} V_\beta$. Then, the natural morphisms $g_{\alpha\beta} : V_{\alpha\beta} \to V_\alpha$ are quasi-\'etale.
        \end{itemize}
    \end{itemize}
    
    We say that $\mathcal{C}$ is a \emph{smooth} orbi-\'etale orbistructure if each $V_\alpha$ is smooth.
\end{defi}

If the pair can be endowed with a smooth orbi-\'etale orbistructure $\mathcal{C}$, then it is a genuine orbifold in the sense of Satake: the pair is locally given by finite quotients of smooth manifolds. In our case, the log smooth pair $(\widehat{X}, \widehat{\Delta})$ is naturally endowed with a smooth orbi-\'etale orbistructure.

\begin{defi}
    Let $(X, \Delta)$ be a pair endowed with a smooth orbi-\'etale orbistructure $\mathcal{C} = \lbrace (U_\alpha, V_\alpha, \eta_\alpha) \rbrace$. We use notation of Definition~\ref{Def orbistructure orbi-etale}. An \emph{orbisheaf} $\mathcal{F}$ is a collection $(\mathcal{F}_\alpha)_\alpha$ of sheaves such that:
    \begin{itemize}
        \item for each $\alpha$, the sheaf $\mathcal{F}_\alpha$ is a coherent $\mathcal{O}_{V_\alpha}$-module,
        \item for each $\alpha, \beta$, there is an isomorphism of $\mathcal{O}_{V_{\alpha\beta}}$-modules $\varphi_{\alpha\beta} : g_{\beta\alpha}^*\mathcal{F}_\beta \to g_{\alpha\beta}^*\mathcal{F}_\alpha$, called \emph{transition function}. They satisfy compatibility conditions on triple intersections. 
    \end{itemize}
We say that $\mathcal{F}$ is an \emph{orbibundle of rank r} (resp. \emph{torsion free}, \emph{reflexive}) if each $\mathcal{F}_\alpha$ is locally free of rank $r$ (resp. torsion free, reflexive). 
\end{defi}

\begin{exem}\textup{\cite[Ex.~2.18 and 2.21]{Dai25}}\label{exemple orbibundles}
Let $V$ be a vector bundle over $\widehat{X}$. One can define:
    \begin{itemize}
        \item the \emph{trivial orbibundle} $\mathcal{O}_{\widehat{X}}$, whose transition functions are denoted $\varphi_{\alpha\beta}$,
        \item the \emph{orbifold tangent bundle} $\mathcal{T}_{(\widehat{X}, \widehat{\Delta})}$, whose transition functions are denoted $\psi_{\alpha\beta}$,
        \item the \emph{orbifold canonical extension} $\mathcal{E}_{(\widehat{X}, \widehat{\Delta})}$.
        \item the \emph{orbifold bundle associated to $V$}, denoted $V^\orb$.
    \end{itemize}
\end{exem}

We consider the line bundles $\mathcal{O}(E_i)$ and $\mathcal{O}(\widehat{\Delta}_j)$, that we endow with metrics $h_i$ and $h'_j$ respectively. There exist numbers $\varepsilon_i > 0$ such that the class:
$$
\pi^*\left[ -(K_X + \Delta) \right] - \sum_i \varepsilon_i  \left[ \textbf{i}\Theta(E_i, h_i) \right]
$$
can be represented by a K\"ahler form $\widehat{\omega}$. Let $\sigma_i$ be a section of $\mathcal{O}(E_i)$ such that $\mathrm{div}\, \sigma_i = E_i$ and $s_j$ a section of $\mathcal{O}(\widehat{\Delta_j})$ such that $\mathrm{div}\, s_j = \widehat{\Delta_j}$. Finally, we fix a volume form $\mathrm{d}V$ on $\widehat{X}$ defined by:
$$
\Ric \mathrm{d}V := \pi^*\omega - \sum_i a_i \textbf{i}\Theta(E_i, h_i) + \sum_j \left( 1 - \frac{1}{m_i} \right) \textbf{i}\Theta(\widehat{\Delta_j}, h'_j). 
$$

One can observe that the $(n,n)$-form:
$ \frac{\mathrm{d}V}{\prod_j |s_j|_{h'_j}^{2 \left( 1 - \frac{1}{m_j} \right)}}$
is orbifold by studying its form on local uniformizing charts. Moreover, a form on $\widehat{X}$ yields naturally an orbifold form on $(\widehat{X}, \widehat{\Delta})$. Hence, we have the following equality of orbifold forms:
$$
\Ric \left( \frac{\mathrm{d}V}{\prod_j |s_j|_{h'_j}^{2 \left( 1- \frac{1}{m_j} \right)}} \right) = \pi^*\omega - \sum_i a_i \textbf{i}\Theta(E_i, h_i).
$$

\subsection{Regularizations of a Kähler--Einstein metric on a resolution}

Assume that $(X, \Delta)$ is endowed with a K\"ahler--Einstein metric $\omega_\Phi := \omega + \ddc \Phi$ where $\Phi$ solves the equation \eqref{KE}. By pulling back this equation to $\widehat{X}$, we get that $\varphi := \pi^*\Phi$ is a locally bounded $\pi^*\omega$-psh function solving the following Monge-Ampère equation satisfied on $(\widehat{X}, \widehat{\Delta})$:
\begin{equation}
    \left(\pi^* \omega + \ddc \varphi \right)^n = e^{-\varphi} \frac{\prod_i |\sigma_i|_{h_i}^{2a_i}}{\prod_j |s_j|_{h'_j}^{2\left(1 - \frac{1}{m_j} \right)}} \mathrm{d}V.
\end{equation}

Thanks to orbifold version of Demailly's regularization results \cite{Wu23}, one can consider a family of functions $\psi_\varepsilon \in \mathcal{C}^\infty_\orb(\widehat{X}, \R)$ such that:
\begin{itemize}
    \item $\psi_\varepsilon \underset{\varepsilon \longrightarrow 0}{\longrightarrow} \varphi$ in $L^1(\widehat{X})$ and in $\mathcal{C}^\infty_{\orb, \mathrm{loc}}(\widehat{X} \backslash |E|)$,
    \item There exists $C > 0$ such that $\Vert \psi_\varepsilon \Vert_{L^\infty(\widehat{X})} \leq C$,
    \item There exists a continuous function $\kappa : [0, 1] \longrightarrow \R_+$ with $\kappa(0) = 0$ such that $\pi^* \omega + \ddc \psi_\varepsilon \geq - \kappa(\varepsilon) \omega_{\orb}$, where $\omega_{\orb}$ is an orbifold metric\footnote{That is a metric which can be expressed around a general point of an irreducible component $\widehat{\Delta}_i$ of $\widehat{\Delta}$ in adapted coordinates:
    $\displaystyle
    \omega_\orb = \frac{\textbf{i} \mathrm{d} z_1 \wedge \mathrm{d}\overline{z_1}}{|z_1|^{2(1 - \frac{1}{m_i})}} + \sum_{2 \leq j \leq n} \textbf{i} \mathrm{d} z_j \wedge \mathrm{d}\overline{z_j}.
    $} on $(\widehat{X}, \widehat{\Delta})$.
\end{itemize}

Now, let us denote $\omega_{t, \varepsilon} := \pi^*\omega + t\widehat{\omega} + \ddc \varphi_{t, \varepsilon}$, where $\varphi_{t,\varepsilon}$ is the potential solving the perturbed Monge--Ampère equation:

$$
\left(\pi^* \omega + t \widehat{\omega} + \ddc \varphi_{t,\varepsilon} \right)^n = e^{-\psi_\varepsilon} \frac{\prod_i \left(|\sigma_i|_{h_i}^{2} + \varepsilon^2 \right)^{a_i}}{\prod_j |s_j|_{h'_j}^{2\left(1 - \frac{1}{m_j} \right)}} N_{t, \varepsilon}\mathrm{d}V,
$$
where $N_{t, \varepsilon}$ is a normalization constant such that:
$$
\left[ \pi^* \omega + t \widehat{\omega} \right]^n = N_{t, \varepsilon} \int_{\widehat{X}} e^{-\psi_\varepsilon} \frac{\prod_i \left(|\sigma_i|_{h_i}^{2} + \varepsilon^2 \right)^{a_i}}{\prod_j |s_j|_{h'_j}^{2\left(1 - \frac{1}{m_j} \right)}} \mathrm{d}V.
$$
The existence of $\varphi_{t, \varepsilon}$ is guaranteed by orbifold version of Yau's theorem \cite{Fau19}. From this result, we deduce moreover that there exists a uniform constant $C$ such that:
$$
\forall t, \varepsilon \in \left[ 0,1 \right], \Vert \varphi_{t, \varepsilon} \Vert_{L^\infty(\widehat{X})} \leq C.
$$

\begin{rema}\label{convergence local smooth orbifold of potentials}
This result can be improved in the following way: the potentials $\varphi_{t, \varepsilon}$ converge to $\varphi$ when $t, \varepsilon$ go to zero thanks to the stability property of the Monge--Ampère operator (see \cite[Thm.~C]{GZ12}). We have a convergence in $\mathcal{C}^\infty_{\orb, \mathrm{loc}}(\widehat{X} \backslash |E|)$. Indeed, the metrics $\omega_{t, \varepsilon}$ have conic singularities along $\widehat{\Delta}$ on $\widehat{X} \backslash |E|$. If we fix $K $ a compact subset of $\widehat{X} \backslash |E|$, then there exists a constant $C_K$ independent of $t, \varepsilon$ such that:
$$
C_K^{-1} \omega_{\orb} \leq \omega_{t, \varepsilon} \leq C_K \omega_{\orb},
$$
where $\omega_{\orb}$ is an orbifold metric on $(\widehat{X}, \widehat{\Delta})$. 
Such bounds are due to \cite[Thm.~2.(ii)]{GP16}. They are independent of $t, \varepsilon$. Indeed, we have from \cite[Prop.4.1]{Gue13} an explicit form of the estimate that blows up near the exceptional divisor. Then, one can show that Evans-Krylov and Schauder estimates can be used on local uniformizations in order to prove that the family $(\varphi_{t, \varepsilon})$ is uniformly bounded in $\mathcal{C}^p_\orb(K)$ for any $p$.\\
\end{rema}

In order to keep notation light, we will denote in the following $|\sigma_i| := |\sigma_i|_{h_i}$ and $|s_j| := |s_j|_{h'_j}$. We set $h_{i, \varepsilon} := \left( |\sigma_i|^2 + \varepsilon^2 \right)^{-1} h_i$, and we denote: 
$$
\theta_{i,\varepsilon} := \textbf{i}\Theta(E_i, h_{i, \varepsilon}) = \ddc \log \left( |\sigma_i|^2 + \varepsilon^2 \right) + \textbf{i}\Theta(E_i, h_{i}). 
$$
One can compute the Ricci curvature of the perturbed orbifold metric $\omega_{t, \varepsilon}$:
\begin{equation}
    \Ric \omega_{t, \varepsilon} = \ddc \psi_\varepsilon - \sum_i a_i \ddc \log \left( |\sigma_i| ^2 + \varepsilon^2 \right) + \Ric \left( \frac{\mathrm{d}V}{\prod_j |s_j|^{2 \left( 1- \frac{1}{m_j} \right)}} \right).
\end{equation}

Hence, we have:
\begin{equation}\label{Perturbed Ric}
\Ric \omega_{t, \varepsilon} = \omega_{t, \varepsilon} + \ddc (\psi_\varepsilon - \varphi_{t, \varepsilon}) - t\widehat{\omega} - \sum_i a_i \theta_{i, \varepsilon}.
\end{equation}

\section{Stability of orbifold bundles on a log resolution}

\subsection{Semistability of the orbifold tangent bundle}

The goal of this section is to show:

\begin{prop}
    With notation as above, the orbifold tangent bundle $\mathcal{T}_{(\widehat{X}, \widehat{\Delta})}$ is semistable with respect to $\pi^*c_1(X, \Delta)$.
\end{prop}

We endow $\mathcal{T}_{(\widehat{X}, \widehat{\Delta})}$ with the hermitian metric $h_{t, \varepsilon}$ induced by the orbifold (1,1)-form $\omega_{t, \varepsilon}$. Let $\mathcal{F}$ be a suborbisheaf of $\mathcal{T}_{(\widehat{X}, \widehat{\Delta})}$ of rank $r$. Up to consider the saturation of $\mathcal{F}$ that has a greater slope than $\mathcal{F}$, we assume that $\mathcal{F}$ is saturated in $\mathcal{T}_{(\widehat{X}, \widehat{\Delta})}$. Hence, one can find an open subset $\widehat{X}^\circ \subset \widehat{X}$ such that $\widehat{X} \backslash \widehat{X}^\circ$ is an analytic subset of codimension at least 2, and $\mathcal{F}_{|\widehat{X}^\circ}$ is a suborbibundle of $\mathcal{T}_{(\widehat{X}, \widehat{\Delta})}$. Hence, we can use the metric $h_{t, \varepsilon}$ in order to compute the slope of $\mathcal{F}$.

\begin{lemm}\label{Kobayashi orbifolde slope}
    We have:
    $$
    \frac{1}{r}\int_{\widehat{X}^\circ} c_1^\orb(\mathcal{F}, h_{t, \varepsilon}) \wedge \omega_{t, \varepsilon}^{n-1} = \mu_{\omega_{t, \varepsilon}}(\mathcal{F}) \underset{t, \varepsilon \longrightarrow 0}{\longrightarrow} \mu_{\pi^* c_1(X, \Delta)}(\mathcal{F}).
    $$
\end{lemm}

\begin{proof}
    The proof is based on the smooth case \cite[Proof of $(\ast\ast)$, p.~166]{Kob87}. We have an orbibundle map:
    $$ 
    \det \mathcal{F} \longrightarrow \bigwedge^r \mathcal{T}_{(\widehat{X}, \widehat{\Delta})}.
    $$
    Let us denote $\mathcal{H} := \bigwedge^r \mathcal{T}_{(\widehat{X}, \widehat{\Delta})} \otimes (\det \mathcal{F})^{-1}$. Hence, the previous map yields a section $s$ of $\mathcal{H}$ which does not vanish on $\widehat{X}^\circ$. We pick $ h_{\det \mathcal{F}}$ a hermitian metric on $\det \mathcal{F}$, and we denote $h := h_{t, \varepsilon}^{\wedge r} \otimes h_{\det \mathcal{F}}^\vee$ the induced metric on $\mathcal{H}$. In uniformizing charts, the actions of local Galois groups $G_\alpha$ are unitary with respect to local hermitian metrics $h_\alpha$ associated to $h$, and the local sections $s_\alpha$ defining $s$ are $G_\alpha$-invariant. Hence, local functions $|s_\alpha|^2_{h_\alpha}$ descend to $\widehat{X}$ and glue together, so that we consider the global orbifold function $|s|^2_h \in \mathcal{C}^\infty_\orb(\widehat{X}, \C)$.\\
    
    We consider a resolution of the ideal generated\footnote{The section $s$ is a map $\mathcal{O}_{\widehat{X}} \longrightarrow \mathcal{H}^\inv$ where $\mathcal{H}^\inv$ is the reflexive sheaf of invariant sections of the orbibundle $\mathcal{H}$. We consider a resolution of the kernel of this map.} by $s$, denoted by $\mu : Z \longrightarrow \widehat{X}$. Now, we analyze how $\mu^*|s|^2_h$ behave along the exceptional divisor $\widetilde{E}$ of $\nu$. Away from $|\widehat{\Delta}|$, the orbibundle $\mathcal{H}$ defines genuinely a vector bundle on $\widehat{X}$, and the section $s$ is locally given by $s = \sum_i f_i e_i$ where $(e_i)$ is a local frame of $\mathcal{H}$. Then, we have $|s|^2_h = \sum_i f_i \overline{f_j} u_{i\overline{j}}$ where $u_{i\overline{j}} := h(e_i, e_j)$.\\
    
	The resolution of $s$ is such that there exists $g$ and $g_i$ local holomorphic functions on $Z$ such that $\mu^* f_i = g_i g$ and the ideal sheaf $(g_i)$ is trivial (here $g$ is a local function defining the Cartier divisor $-m\widetilde{E}$). Hence, we have:
	$$
	\mu^* \log |s|^2_h = \log |g|^2 + \log (\sum_i g_i \overline{g_j} \mu^* u_{i\overline{j}}),
	$$
	where the latter term in the right hand side is smooth. Near a point of the strict transform of $\widehat{\Delta}$, we have:
	$$ \mu^*\log |s|^2_h = \log (\sum_i g_i \overline{g_j} \mu^* u_{i\overline{j}}),
	$$
	where this term is not necessarily smooth on $Z$ as it is the pullback of an orbifold function on $\widehat{X}$.\\
	 
	 Let us endow the line bundle $\mathcal{O}(-m \widetilde{E})$ with a hermitian metric $h_{-m \widetilde{E}}$. Hence, thanks to the cocycles identities satisfied by the local functions defining $-m \widetilde{E}$, there exists a continuous function $\Psi \in \mathcal{C}^0(Z)$ such that the following equality of current over $Z$ holds:
	 
	 $$
	 \ddc \mu^* \log |s|^2_h = -m \left[ \widetilde{E} \right] - \textbf{i}\Theta(-m \widetilde{E}, h_{-m\widetilde{E}}) + \ddc \Psi.
	 $$
    Hence, we get:
    \begin{equation}
    \begin{split}    
    \int_{\widehat{X}^\circ} \ddc \log |s|_h^2 \wedge \omega_{t, \varepsilon}^{n-1} 
    & = \int_{Z} \mu^*\ddc \log |s|_h^2 \wedge \mu^*\omega_{t, \varepsilon}^{n-1} \\
    & = -m \int_{|\widetilde{E}|} \mu^*\omega_{t, \varepsilon}^{n-1} - \int_{Z} \textbf{i}\Theta(-m \widetilde{E}, h_{-m\widetilde{E}}) \wedge \mu^*\omega_{t, \varepsilon}^{n-1} \\
    & \qquad + \int_{Z } \ddc \Psi \wedge \mu^*\omega_{t, \varepsilon}^{n-1}.
    \end{split}
    \end{equation}
    Since $\mu^*\omega_{t, \varepsilon}^{n-1}$ is $\widetilde{E}$-exceptional, we have:
    $$
    \int_{|\widetilde{E}|} \mu^*\omega_{t, \varepsilon}^{n-1} = \int_{Z} \textbf{i}\Theta(-m \widetilde{E}, h_{-m\widetilde{E}}) \wedge \mu^*\omega_{t, \varepsilon}^{n-1} = 0.
    $$
    Moreover, we have as a consequence of the definition of $\ddc \Psi \wedge \mu^*\omega_{t, \varepsilon}^{n-1}$:
    $$
    \int_{Z } \ddc \Psi \wedge \mu^*\omega_{t, \varepsilon}^{n-1} = 0.
    $$
    Note that the forms $\omega_{t, \varepsilon}$ are singular, but they have bounded potentials $\varphi_{t, \varepsilon}$, so that the computation of these integrals is purely cohomological thanks to classical Bedford--Taylor results (apply for instance \cite[Thm.~2.1]{Dar19} with \cite[Thm.~3.18]{GZ17}).\\
    
    Now, we deduce by considering local trivializations of the orbibundle $\det \mathcal{F}$ and $\bigwedge^r \mathcal{T}_{(\widehat{X}, \widehat{\Delta})}$ that $\ddc \log |s|^2_h = c_1^\orb(\det \mathcal{F}, h_{\det \mathcal{F}}) - c_1^\orb(\det \mathcal{F}, h^{\wedge r})$ on $\widehat{X}^\circ$. We deduce the sought equality by integrating over $\widehat{X}^\circ$.\\
    
    Because of boundedness of potentials $\varphi_{t, \varepsilon}$, the computation of the slope is purely cohomological thanks to classical Bedford--Taylor results (see e.g. \cite[Lemma.~2.2]{Dar19}) so that:
    $$
    \mu_{\omega_{t, \varepsilon}}(\mathcal{F}) = \frac{1}{r} c_1(\mathcal{F}) \cdot \left[ \pi^*\omega + t\widehat{\omega} \right]^{n-1} \underset{t, \varepsilon \longrightarrow 0}{\longrightarrow} \mu_{\pi^* c_1(X, \Delta)}(\mathcal{F}).
    $$
\end{proof}

Now, let us denote by $\beta_{t, \varepsilon} \in \mathcal{C}^\infty_{(0,1)}(\widehat{X}, \mathcal{H}om(\mathcal{F}, \mathcal{F}^\perp))$ the second fundamental form obtained from the $\mathcal{C}^\infty$-splitting $\mathcal{T}_{(\widehat{X}, \widehat{\Delta})} = \mathcal{F} \oplus \mathcal{F}^\perp$ given by the metric $h_{t, \varepsilon}$. We get that:
$$
\textbf{i}\Theta (F, h_{t,\varepsilon}) =  \mathrm{pr}_F \left(\textbf{i}\Theta(\mathcal{T}_{(\widehat{X}, \widehat{\Delta})}, h_{t, \varepsilon})_{|F} \right) + \textbf{i}\beta_{t, \varepsilon} \wedge \beta_{t, \varepsilon}^*.
$$

Hence, we have:

\begin{align*}
    \mu_{\omega_{t, \varepsilon}}(\mathcal{F}) 
    & = \frac{1}{r} \int_{\widehat{X}} \tr_\End \left( \textbf{i}\Theta(F, h_{t,\varepsilon}) \right) \wedge \omega_{t, \varepsilon}^{n-1} \\
    & = \frac{1}{r} \int_{\widehat{X}} \tr_\End \left(  \mathrm{pr}_F \textbf{i}\Theta(\mathcal{T}_{(\widehat{X}, \widehat{\Delta})}, h_{t, \varepsilon})_{|F} \right) \wedge \omega_{t, \varepsilon}^{n-1} \\
    & \qquad + \frac{1}{r} \int_{\widehat{X}} \tr_\End \left(\textbf{i} \beta_{t, \varepsilon} \wedge \beta_{t, \varepsilon}^* \right) \wedge \omega_{t, \varepsilon}^{n-1}.
\end{align*}

Now, we introduce musical isomorphisms: if we have an orbifold K\"ahler form $\Omega$ inducing a metric $h$ and an orbifold $(1,1)$-form $\gamma$, we define $\sharp_\Omega \gamma$ (or $\sharp \gamma$ if there is no ambiguity) as the section of $\End(\mathcal{T}_{(\widehat{X}, \widehat{\Delta})})$ satisfying the equation:
$$
\tr_\End (\sharp \gamma) \Omega^n = n \gamma \wedge \Omega^{n-1}.
$$

Hence, we have:

\begin{align*}
    \mu_{\omega_{t, \varepsilon}}(\mathcal{F}) = \frac{1}{nr} \int_{\widehat{X}} & \tr_\End \left( \textbf{i} \mathrm{pr}_F \sharp \Ric {\omega_{t, \varepsilon}}_{|F} \right) \omega_{t, \varepsilon}^n \\
    & + \frac{1}{r} \int_{\widehat{X}} \tr_\End \left(\textbf{i}\beta_{t, \varepsilon} \wedge \beta_{t, \varepsilon}^* \right) \wedge \omega_{t, \varepsilon}^{n-1}.
\end{align*}

We deduce by \eqref{Perturbed Ric}:

\begin{align*}\label{Slope of the subsheaf}
    \tag*{$(\ast)$} \mu_{\omega_{t, \varepsilon}}(\mathcal{F}) & = \frac{1}{ nr} \underbrace{\int_{\widehat{X}} \tr_\End (\mathrm{pr}_F \sharp {\omega_{t, \varepsilon}}_{|F}) \omega_{t, \varepsilon}^n}_{\textbf{(I)}} -t \frac{1}{ nr} \underbrace{\int_{\widehat{X}} \tr_\End (\mathrm{pr}_F \sharp \widehat{\omega}_{|F}) \omega_{t, \varepsilon}^n}_{\textbf{(II)}} \\
    & \qquad + \frac{1}{ nr} \underbrace{\int_{\widehat{X}} \tr_\End (\mathrm{pr}_F \sharp \ddc (\psi_\varepsilon - \varphi_{t,\varepsilon})_{|F}) \omega_{t, \varepsilon}^n}_{\textbf{(III)}} \\
    & \qquad - \frac{1}{ nr} \underbrace{\int_{\widehat{X}} \tr_\End \left( \mathrm{pr}_F \sharp \sum_i a_i {\theta_{i,\varepsilon}}_{|F} \right) \omega_{t, \varepsilon}^n}_{\textbf{(IV)}} \\
    & \qquad + \frac{1}{ r} \underbrace{\int_{\widehat{X}} \tr_\End \left(\textbf{i}\beta_{t, \varepsilon} \wedge \beta_{t, \varepsilon}^* \right) \wedge \omega_{t, \varepsilon}^{n-1}}_{\textbf{(V)}}.
\end{align*}

\subsubsection{Analysis of the different terms}\label{Analysis of different terms}

We analyze here the different terms appearing in the above equality.\\

\textbf{Term (I) :} We have:
$$
\frac{1}{ nr} \int_{\widehat{X}} \tr_\End (\mathrm{pr}_F \sharp {\omega_{t, \varepsilon}}_{|F}) \omega_{t, \varepsilon}^n = \frac{1}{n} \int_{\widehat{X}} \omega_{t, \varepsilon}^n.
$$
But thanks to the boundedness of $\varphi_{t, \varepsilon}$, we get by classical Bedford--Taylor results (see e.g. \cite[Thm.~2.1]{Dar19}):
$$
\frac{1}{n} \int_{\widehat{X}} \omega_{t, \varepsilon}^n = \frac{1}{n} \left[\pi^* \omega + t\widehat{\omega} \right]^n \underset{t, \varepsilon \longrightarrow 0}{\longrightarrow} \frac{1}{n} \left[\pi^* \omega\right]^n = \mu_{\pi^*\omega} (\mathcal{T}_{(\widehat{X}, \widehat{\Delta})}).\\
$$

\textbf{Term (II) :} This term is positive, since $\sharp \widehat{\omega}$ is a positive endomorphism.\\

\textbf{Term (III) :} Let us show that this term converges to zero when $t, \varepsilon$ go to 0. We adapt the arguments from \cite[Term I~-~Sect.2.2]{DGP24}.

\emph{Bound near $|E + \widehat{\Delta}|$ :} We use the following lemma.

\begin{lemm}\label{Bound near the divisor}
    With notation as above, for any $\delta > 0$, there exists a neighborhood $U_\delta$ of $|E + \widehat{\Delta}|$ such that:
    $$
    \forall t, \varepsilon > 0, \int_{U_\delta} (\omega_{\psi_\varepsilon} + \omega_{t, \varepsilon} + \omega_\orb) \wedge \omega_{t, \varepsilon}^{n-1} < \delta,
    $$
    where $\omega_{\psi_\varepsilon} := \pi^*\omega + t \widehat{\omega} + \ddc \psi_\varepsilon$ and $\omega_\orb$ is an orbifold metric on $(\widehat{X}, \widehat{\Delta})$.
\end{lemm}
We need to consider here $\omega_{\psi_\varepsilon} + \omega_{t, \varepsilon} + \omega_\orb$ in order to get positive orbifold metrics. For the convenience of the reader, we remind the main arguments to show this lemma. Let us consider a family of cutoff functions $\Xi_\delta: \R \longrightarrow \R$ such that
$$
\Xi_\delta(x) =
\left\lbrace 
\begin{array}{cc}
     0 & \text{ if } x \leq \delta^{-1} \\
     1 & \text{ if } x \geq 1 + \delta^{-1}
\end{array}
\right.
.
$$
We can moreover assume that the derivative and the hessian of $\Xi_\delta$ are uniformly bounded independently of $\delta$. We consider the functions $\rho_\delta := \Xi_\delta \left( \log \log \frac{1}{|s_{E + \widehat{\Delta}}|^2} \right)$. Let $U_\delta$ be an open subset included in $\left\lbrace |s_{E + \widehat{\Delta}}|^2 < e^{-e^{1 + \delta^{-1}}} \right\rbrace$.

Using the Chern--Levine--Nirenberg inequality, we have:
\begin{equation}
    \int_{U_\delta} \left( \omega_{\psi_\varepsilon} + \omega_{t, \varepsilon} + \omega_\orb \right) \wedge \omega_{t, \varepsilon}^{n-1} \leq \int_{\widehat{X}} \rho_\delta \left( \omega_{\psi_\varepsilon} + \omega_{t, \varepsilon} + \omega_\orb \right) \wedge \omega_{t, \varepsilon}^{n-1}
\end{equation}
and the latter integral is equal to:
\begin{equation}
    2\underbrace{\int_{\widehat{X}} \rho_\delta  \left( \pi^*\omega + t \widehat{\omega} \right) \wedge \omega_{t, \varepsilon}^{n-1}}_{\textbf{(A)}} + \underbrace{\int_{\widehat{X}} \rho_\delta \ddc \left( \psi_\varepsilon + \varphi_{t, \varepsilon} \right) \wedge \omega_{t, \varepsilon}^{n-1}}_{\textbf{(B)}} + \underbrace{\int_{\widehat{X}} \rho_\delta \; \omega_\orb \wedge \omega_{t, \varepsilon}^{n-1}}_{\textbf{(C)}}.
\end{equation}

Let us deal with term \textbf{(A)}. On a neighborhood $U$ of $|E + \widehat{\Delta}|$, the metric $\pi^* \omega + t \widehat{\omega}$ is bounded by $\omega_{E + \widehat{\Delta}}$, a metric with Poincar\'e singularities along $|E + \widehat{\Delta}|$ up to a constant depending only on $U$. Moreover, the function $\rho_\delta$ is supported in the open set $V_\delta := \lbrace |s_{E + \widehat{\Delta}}|^2 < e^{-e^{\delta^{-1}}} \rbrace$, and it takes values in $\left[ 0,1 \right]$, so this term is bounded by:
$$
C \int_{V_\delta} \omega_{E + \widehat{\Delta}} \wedge \omega_{t, \varepsilon}^{n-1},
$$
where $C$ is a constant depending on the neighborhood $U$ (in particular, it is independent of $\delta$).\\

Let us deal with term \textbf{(B)}. By Stokes' theorem, we have:
$$
 \int_{\widehat{X}} \rho_\delta \; \ddc \left( \psi_\varepsilon + \varphi_{t, \varepsilon} \right) \wedge \omega_{t, \varepsilon}^{n-1} = \int_{\widehat{X}} \left( \psi_\varepsilon + \varphi_{t, \varepsilon} \right)\ddc \rho_\delta  \wedge \omega_{t, \varepsilon}^{n-1}.
$$
But if we denote $u := \log \log \frac{1}{|s_{E + \widehat{\Delta}}|^2}$, we have in local coordinates:
$$
\ddc \rho_\delta = \sum_{j, \overline{k}} \left(\frac{\partial^2 u }{\partial z_j \partial \overline{z_k}} \Xi'_\delta \circ u + \frac{\partial u }{\partial z_j}\frac{\partial u }{\partial \overline{z_k}} \Xi''_\delta \circ u \right)\; \textbf{i}\mathrm{d} z_j \wedge \mathrm{d}\overline{z_k}.
$$
The latter term is bounded from above up to a constant independent of $\delta$ by $\omega_{E + \widehat{\Delta}}$. Moreover, $\ddc \rho_\delta$ is supported in $V_\delta$, and $\psi_\varepsilon$ and $\varphi_{t, \varepsilon}$ are uniformly bounded in $t, \varepsilon$. Hence, the term \textbf{(B)} is bounded by:
$$
C\int_{V_\delta} \omega_{E + \widehat{\Delta}} \wedge \omega_{t, \varepsilon}^{n-1}.
$$
Let us deal with term \textbf{(C)}. We show that $\omega_\orb$ is bounded by $\omega_{E + \widehat{\Delta}}$ by comparing local forms of the metrics, so that:
$$
\int_{\widehat{X}} \rho_\delta \; \omega_\orb \wedge \omega_{t, \varepsilon}^{n-1} \leq \int_{V_\delta} \omega_\orb \wedge \omega_{t, \varepsilon}^{n-1} \leq C \int_{V_\delta} \omega_{E + \widehat{\Delta}} \wedge \omega_{t, \varepsilon}^{n-1}.
$$
Now, by induction, we can show similarly the existence of $\delta' > 0$ depending on $\delta$ such that:
$$
\int_{V_\delta} \omega_{E + \widehat{\Delta}} \wedge \omega_{t, \varepsilon}^{n-1}  \leq C \int_{V_{\delta'}} \omega_{E + \widehat{\Delta}}^n.
$$
The quantity $\delta'$ depends on $\delta$, but one can choose $U_\delta$ sufficiently small in order to have $\displaystyle \int_{V_{\delta'}} \omega_{E + \widehat{\Delta}}^n < \frac{\delta}{C^2}$. This shows the bound near $|E + \widehat{\Delta}|$. \qed \\

Now, we can observe that, as hermitian endomorphisms, we have\footnote{This is due to the fact that $\omega_{t, \varepsilon} \geq 0$ and $\omega_{\psi_\varepsilon} + \frac{1}{2} \omega_\orb \geq \left( \frac{1}{2} -\kappa(\varepsilon) \right) \omega_\orb$ as a consequence of properties of $\psi_\varepsilon$. This holds for $t, \varepsilon$ smaller than constants depending only on the function $\kappa$.}:
$$
\pm \sharp \ddc(\psi_\varepsilon - \varphi_{t, \varepsilon}) \leq \sharp \left( \omega_{\psi_\varepsilon} + \omega_{t, \varepsilon} + \omega_\orb \right).
$$
Since the operations $\mathrm{pr}_F (.)_{|F}$, $\tr_\End$ preserve positivity (hence inequalities), we have:
$$
\left| \int_{U_{\delta}} \tr_\End (\mathrm{pr}_F \sharp \ddc (\psi_\varepsilon - \varphi_{t,\varepsilon})_{|F}) \omega_{t, \varepsilon}^n \right| \leq \delta.
$$

\emph{Bound on a compact subset of $\widehat{X} \backslash |E + \widehat{\Delta}|$ :}
Let $K$ be a compact subset of $\widehat{X} \backslash |E + \widehat{\Delta}|$. The potentials $\varphi_{t, \varepsilon}, \psi_\varepsilon$ converge to $\varphi$ in $\mathcal{C}^\infty(K)$. Hence, there exists $\eta > 0$ such that: 
$$
\forall t, \varepsilon < \eta, \Vert \psi_\varepsilon - \varphi_{t, \varepsilon} \Vert_{\mathcal{C}^2(K)} < \delta.
$$
For such $t, \varepsilon$, we have the bound on $K$: $\pm \ddc \left( \psi_\varepsilon - \varphi_{t, \varepsilon} \right) < \delta \widehat{\omega}$, hence there exists a constant $C > 0$ independent of $t, \varepsilon$ such that:
\begin{equation}\label{Bound on compact subsets}
\left| \int_K \tr_\End (\mathrm{pr}_F \sharp \ddc (\psi_\varepsilon - \varphi_{t,\varepsilon})_{|F}) \omega_{t, \varepsilon}^n \right| \leq \delta \int_K \widehat{\omega} \wedge \omega_{t, \varepsilon}^{n-1} < C \delta
\end{equation}

\emph{Computation of the limit :} We pick $\delta > 0$. Then, thanks to Lemma~\ref{Bound near the divisor}, there exists a neighborhood $U_\delta$ of $|E + \widehat{\Delta}|$ such that:
$$
\left| \int_{U_{\delta}} \tr_\End (\mathrm{pr}_F \sharp \ddc (\psi_\varepsilon - \varphi_{t,\varepsilon})_{|F}) \omega_{t, \varepsilon}^n \right| \leq \frac{\delta}{2}.
$$
Hence, using \eqref{Bound on compact subsets} on the compact subset $\widehat{X} \backslash U_\delta \subseteq \widehat{X} \backslash |E + \widehat{\Delta}|$, there exists $\eta > 0$ such that:
$$
\forall t, \varepsilon < \eta, \left| \int_{\widehat{X} \backslash U_{\delta}} \tr_\End (\mathrm{pr}_F \sharp \ddc (\psi_\varepsilon - \varphi_{t,\varepsilon})_{|F}) \omega_{t, \varepsilon}^n \right| \leq \frac{\delta}{2}.
$$
By combining these two inequalities, we get that the term \textbf{(III)} is less than $\delta$ for $t, \varepsilon$ sufficiently small. In other words:
$$
\lim_{t, \varepsilon \to 0} \int_{\widehat{X}} \tr_\End (\mathrm{pr}_F \sharp \ddc (\psi_\varepsilon - \varphi_{t,\varepsilon})_{|F}) \omega_{t, \varepsilon}^n = 0.\\
$$

\textbf{Term (IV) :}
We denote by $D'_i$ the $(1,0)$ part of the Chern connection of $h_i$. We have:

$$
\theta_{i,\varepsilon} = \underbrace{\frac{\varepsilon^2}{(|\sigma_i|^2 + \varepsilon^2)^2}|D'_i\sigma_i|^2}_{=:\beta_i} + \underbrace{\frac{\varepsilon^2}{|\sigma_i|^2 + \varepsilon^2} \theta_i}_{=: \gamma_i}
$$
Let us show that:
$$
\int_{\widehat{X}} \tr_\End(\mathrm{pr}_F{\theta_{i, \varepsilon}}_{|F}) \wedge \omega_{t, \varepsilon}^{n} \underset{t, \varepsilon \longrightarrow 0}{\longrightarrow} 0.
$$

In order to keep notation light, we drop the index $i$ referring to the irreducible component $E_i$ of $E$. First, we deal with $\gamma$. There exists a constant $C > 0$ independent of $t, \varepsilon$ such that $\pm \theta \leq C \widehat{\omega}$. Then, we have:
\begin{align*}
    \left| \int_{\widehat{X}} \tr_\End(\mathrm{pr}_F \sharp \gamma_{|F}) \omega_{t, \varepsilon}^n \right| & \leq C \int_{\widehat{X}} \frac{\varepsilon^2}{|\sigma|^2 + \varepsilon^2} \tr_\End(\mathrm{pr}_F \sharp \widehat{\omega}_{|F}) \omega_{t, \varepsilon}^n \\
    & = Cn \int_{\widehat{X}} \frac{\varepsilon^2}{|\sigma|^2 + \varepsilon^2} \widehat{\omega} \wedge \omega_{t, \varepsilon}^{n-1}
\end{align*}
If we denote $V_{\varepsilon} := \lbrace |\sigma|^2 < \varepsilon \rbrace$, then:
\begin{align*}
    \int_{\widehat{X} \backslash V_{\varepsilon}} \frac{\varepsilon^2}{|\sigma|^2 + \varepsilon^2} \widehat{\omega} \wedge \omega_{t, \varepsilon}^{n-1} 
    & \leq \varepsilon \int_{\widehat{X} \backslash V_{\varepsilon}} \widehat{\omega} \wedge \omega_{t, \varepsilon}^{n-1} \\
    & \leq \varepsilon \int_{\widehat{X}} \widehat{\omega} \wedge \omega_{t, \varepsilon}^{n-1} = \varepsilon \left[ \widehat{\omega} \right] \cdot \left[ \pi^* \omega + t\widehat{\omega} \right]^{n-1} 
\end{align*}

On the other hand, for any $\delta > 0$, there exists $\eta > 0$ such that for any $0 < \varepsilon < \eta$, we have by using the same computations as the term \textbf{(III)}:
\begin{align*}
    \int_{V_{ \varepsilon}} \frac{\varepsilon^2}{|\sigma|^2 + \varepsilon^2} \widehat{\omega} \wedge \omega_{t, \varepsilon}^{n-1} & \leq \int_{V_{ \varepsilon}} \widehat{\omega} \wedge \omega_{t, \varepsilon}^{n-1} \leq C' \int_{U_\delta} \omega_{E + \widehat{\Delta}}^n \leq C'\delta
\end{align*}
where $C' > 0$ is a constant depending on $\delta$.
This means that:
$$
\int_{V_{\varepsilon}} \frac{\varepsilon^2}{|\sigma|^2 + \varepsilon^2} \widehat{\omega} \wedge \omega_{t, \varepsilon}^{n-1} \underset{\varepsilon \longrightarrow 0}{\longrightarrow} 0.
$$
Hence, we have:
$$
\int_{\widehat{X}} \tr_\End(\mathrm{pr}_F \sharp \gamma_{|F}) \omega_{t, \varepsilon}^n \underset{\varepsilon \longrightarrow 0}{\longrightarrow} 0.
$$

Now we deal with the term $\beta$. We observe that $\sharp \beta$ is a positive endomorphism, hence we have:
$$
0 \leq \int_{\widehat{X}} \tr_\End(\mathrm{pr}_F \sharp \beta_{|F}) \omega_{t, \varepsilon}^n \leq n \int_{\widehat{X}} \beta \wedge \omega_{t, \varepsilon}^{n-1}.
$$
Since $\beta = \theta_\varepsilon - \gamma$, we infer:
$$
\int_{\widehat{X}} \beta \wedge \omega_{t, \varepsilon}^{n-1} = \int_{\widehat{X}} \theta_\varepsilon \wedge \omega_{t, \varepsilon}^{n-1} - \int_{\widehat{X}} \frac{\varepsilon^2}{|\sigma|^2 + \varepsilon^2} \theta \wedge \omega_{t, \varepsilon}^{n-1}.
$$
But the first integral in the right hand side is equal to $c_1(E_i) \cdot \left[ \pi^* \omega + t \widehat{\omega} \right]^{n-1} \underset{t \longrightarrow 0}{\longrightarrow} 0$, and the second one converges to zero when $\varepsilon$ goes to zero thanks to the computations about $\gamma$.
We conclude that term \textbf{(IV)} converges to zero when $t, \varepsilon$ go to zero.\\

\textbf{Term (V) :} This term is negative by computations \cite[Proof of 5.8.2]{Kob87}.\\

We conclude by taking the limit when $t, \varepsilon$ go to zero in \ref{Slope of the subsheaf}, so that:
$$
\mu_{\pi^*c_1(X, \Delta)}(\mathcal{F}) \leq \mu_{\pi^*c_1(X, \Delta)}(\mathcal{T}_{(\widehat{X}, \widehat{\Delta})}).
$$

We deduce that $\mathcal{T}_{(\widehat{X}, \widehat{\Delta})}$ is semistable with respect to $\pi^*c_1(X, \Delta)$.

\subsection{Semistability of an extension of the orbifold tangent bundle}

\subsubsection{Extensions of the orbifold tangent bundle}

We recall that the pair $(\widehat{X}, \widehat{\Delta})$ is endowed with a smooth orbi-\'etale orbistructure $\mathcal{C}$ and some geometric orbibundles are defined. See Definition~\ref{Def orbistructure orbi-etale} and Example~\ref{exemple orbibundles} for the notation used. Let us define first the orbibundle $\widehat{\mathcal{E}}$ that plays the role of the extension of $\mathcal{T}_{(\widehat{X}, \widehat{\Delta})}$ by the trivial orbibundle $\mathcal{O}_{\widehat{X}}$ induced by the class $\pi^*(-(K_X + \Delta))$. There exists an integer $m$ such that $-m(K_X + \Delta)$ is Cartier. The pullback of this divisor by the morphism $\pi$ induces an orbibundle of rank one $\mathcal{L}$. Assume that the orbibundle is trivialized on each local uniformizing chart of $\mathcal{C}$, so that we have on intersections functions $\Phi_{\alpha\beta} : V_{\alpha\beta} \longrightarrow \C^*$. Then if we consider $\omega_{0, \alpha\beta} := \frac{1}{m} \frac{\mathrm{d}\Phi_{\alpha\beta}}{\Phi_{\alpha\beta}}$, one can show that it satisfies a cocycle relation and it defines an orbibundle $\widehat{\mathcal{E}}$ that is an extension of $\mathcal{T}_{(\widehat{X}, \widehat{\Delta})}$ by $\mathcal{O}_{\widehat{X}}$ (see \cite[Prop.~2.20]{Dai25} for the definition of an extension of orbibundles). The goal of this section is to show:

\begin{prop}
    With notation as above, the orbibundle $\widehat{\mathcal{E}}$ is semistable with respect to $\pi^*c_1(X, \Delta)$.
\end{prop}

We use the same deformation arguments as in \cite[Proof of Thm.~11]{DGP24}. We consider the $(1,1)$-form $\frac{1}{1+t}(\pi^*\omega + t \widehat{\omega})$ that induces an orbifold $(1,1)$-form. Let us define a family of extensions $\mathcal{V}_t$ induced by these classes (for $t = 0$, we have $\mathcal{V}_0 = \widehat{\mathcal{E}}$). In this case, the deformation has a special form, since the class itself is obtained as a deformation of $\pi^*c_1(X, \Delta)$ by the cocycles defining $\mathcal{O}(E_i)$. Indeed, we have:
$$
\pi^*\omega + t\widehat{\omega} = (1+t)\pi^* \omega -t\sum_i \varepsilon_i \textbf{i}\Theta(E_i, h_i).
$$

We first assume that the smooth orbi-\'etale orbistructure provides a covering of $\widehat{X}$ by sufficiently small open subsets so that we can consider a family of local meromorphic functions $(b_{i,\alpha})$ that describes the Cartier divisor $E_i$. We denote by $a_{i, \alpha\beta} := \frac{b_{i,\alpha}}{b_{i, \beta}}$ the cocycles of the associated line bundle.
Hence by pulling back each local functions through local uniformizations, the transition maps of the induced orbifold line bundle $\mathcal{O}(E_i)^\orb$ on intersections of local uniformizations are denoted $A_{i, \alpha\beta} := g_{\alpha\beta}^*\eta_\alpha^*a_{i, \alpha\beta}$, and we have on $V_{\alpha\beta}$:
$$
A_{i, \alpha\beta} = \frac{g_{\alpha\beta}^* B_{i,\alpha}}{g_{\beta\alpha}^* B_{i, \beta}},
$$
where $B_{i, \alpha} := \eta_\alpha^*b_{i, \alpha}$. Hence, by differentiating its logarithm, we have the following equality of sections:
$$
\frac{\mathrm{d}A_{i, \alpha\beta}}{A_{i, \alpha\beta}} =  g_{\alpha\beta}^*\frac{\mathrm{d} B_{i,\alpha}}{B_{i,\alpha}}  \circ \psi_{\alpha\beta} - \varphi_{\alpha\beta} \circ g_{\beta\alpha}^*\frac{\mathrm{d}B_{i, \beta}}{B_{i,\beta}}.
$$

For $c = (c_1, \cdots, c_k) \in \C^k$, we can consider the following cocycles:
$$
\omega_{c, \alpha\beta} := \omega_{0, \alpha\beta} - \sum_{1 \leq i \leq k}  c_i \frac{\mathrm{d}A_{i, \alpha\beta}}{A_{i, \alpha\beta}},
$$
where $\omega_{0, \alpha\beta}$ are the holomorphic 1-forms induced by $-\pi^*(K_X + \Delta)$ defined previously. By taking $c(t) := \frac{t}{1+t}(\varepsilon_1, \cdots, \varepsilon_k)$, we get the sought extension $\mathcal{V}_t$.

\subsubsection{Inclusion of suborbisheaf}\label{Inclusion of suborbisheaf}

Let $\mathcal{F} \subseteq \widehat{\mathcal{E}}$ be a suborbisheaf, whose transition maps of orbisheaves are denoted $(\Sigma_{\alpha\beta} : g_{\beta\alpha}^* \mathcal{F}_\beta \longrightarrow g_{\alpha\beta}^* \mathcal{F}_\alpha)$. Let us show that for any $t$ sufficiently small, there is an inclusion of orbisheaf
$$
\mathcal{F} \subseteq \mathcal{V}_t(E) := \mathcal{V}_t \otimes \mathcal{O}(E)^\orb,
$$
where $E = \sum_i E_i$ is the reduced exceptional divisor. For each local uniformizing chart, there is an inclusion  of sheaves $\mathcal{F}_\alpha \hookrightarrow \mathcal{E}_\alpha = \mathcal{O}_{V_\alpha} \oplus \mathcal{T}_{V_\alpha}$. This can be understood through two morphisms
\begin{align*}
    p_\alpha & : \mathcal{F}_\alpha \longrightarrow \mathcal{O}_{V_\alpha} \\
    q_\alpha & : \mathcal{F}_\alpha \longrightarrow \mathcal{T}_{V_\alpha}.  
\end{align*}

These morphisms satisfy compatibility conditions on intersections of the orbistructure:
\begin{align*}
    g_{\alpha\beta}^* p_\alpha \circ \Sigma_{\alpha\beta} & = \varphi_{\alpha\beta} \circ g_{\beta\alpha}^* p_\beta + \omega_{0, \alpha\beta} \circ g_{\beta\alpha}^* q_\beta \\
    g_{\alpha\beta}^* q_\alpha \circ \Sigma_{\alpha\beta} & = \psi_{\alpha\beta} \circ g_{\beta\alpha}^* q_\beta.
\end{align*}

We define for $t \neq -1$:
\begin{align*}
    p_{t, \alpha} & = p_\alpha + \frac{t}{1+t} \sum_i \varepsilon_i \frac{\mathrm{d}B_{i, \alpha}}{B_{i, \alpha}} \circ q_\alpha \\
    q_{t, \alpha} & = q_\alpha
\end{align*}

One can check that we have compatibility conditions on intersections on local uniformizations:
\begin{align*}
    g_{\alpha\beta}^* p_{t,\alpha} \circ \Sigma_{\alpha\beta} & = \varphi_{\alpha\beta} \circ g_{\beta\alpha}^* p_{t, \beta} + \omega_{c(t), \alpha\beta} \circ g_{\beta\alpha}^* q_{t, \beta} \\
    g_{\alpha\beta}^* q_{t, \alpha} \circ \Sigma_{\alpha\beta} & = \psi_{\alpha\beta} \circ g_{\beta\alpha}^* q_{t, \beta}.
\end{align*}

Hence, the collection $(p_{t, \alpha}, q_{t, \alpha})$ induces a morphism of orbisheaves $\mathcal{F} \to \mathcal{V}_t(E)$. Moreover, this morphism $(p_{t, \alpha}, q_{t, \alpha})$ remains injective for any $t$ in a neighborhood of $0$ and any $\alpha$.

\subsubsection{Curvature of the deformation}

We endow the orbibundle with the Chern connection $\nabla_{t, \varepsilon}$ of the metric $h_{\mathcal{V}_t, \varepsilon} = h_{\mathcal{O}_{\widehat{X}}} \oplus h_{t, \varepsilon}$. 
Since the metrics $\omega_{t, \varepsilon}$ are orbifold, they induce on each local uniformization $\eta_\alpha : V_\alpha \to U_\alpha$ a section $\omega_{t, \varepsilon, \alpha} \in \mathcal{C}^\infty_{(0,1)}(V_\alpha, \mathcal{T}_{V_\alpha}^\vee)$. If we denote by $\beta_{t, \varepsilon, \alpha} := \frac{1}{1+t}\omega_{t, \varepsilon, \alpha}$, the connection $\nabla_{t, \varepsilon}$ can be decomposed in the following way with respect to the splitting $\mathcal{O}_{\widehat{X}} \oplus \mathcal{T}_{(\widehat{X}, \widehat{\Delta})}$:
$$
\nabla_{t, \varepsilon, \alpha} =
\begin{bmatrix}
    d_\alpha & \beta_{t, \varepsilon,\alpha} \\
    - \beta_{t, \varepsilon,\alpha}^* & D_{t, \varepsilon, \alpha}
\end{bmatrix},
$$
where $(d_\alpha)$ and $(D_{t, \varepsilon, \alpha})$ are the orbifold Chern connections associated to the trivial metric $h_{\mathcal{O}_{\widehat{X}}}$ and $h_{t, \varepsilon}$ respectively.\\

From now, we remove the dependence on $\alpha$ in order to keep notation light, but all the following computations are done in a local uniformizing chart. The curvature tensor associated to this connection is:
\begin{equation}
    \Theta(\mathcal{V}_t, \nabla_{t, \varepsilon}) =
\begin{bmatrix}
    - \beta_{t, \varepsilon } \wedge \beta_{t, \varepsilon }^* & d  \circ \beta_{t, \varepsilon } + \beta_{t, \varepsilon } \circ D  \\
    - \beta_{t, \varepsilon }^* \circ d  - D  \circ \beta_{t, \varepsilon }^* & \Theta(\mathcal{T}_{V }, h_{t, \varepsilon}) - \beta_{t, \varepsilon }^* \wedge \beta_{t, \varepsilon }
\end{bmatrix}
\end{equation}

Let us compute the terms appearing in this tensor.
First, we have in coordinates:
\begin{equation}\label{Beta local}
    \beta_{t, \varepsilon} = \frac{1}{1+t} \sum_{p,q} \omega_{p\overline{q}}\; \mathrm{d}\overline{z_q} \otimes \left( \frac{\partial}{\partial z^p} \right)^\vee.
\end{equation}
Thanks to the following identity for any section $s \in H^0(V, \mathcal{O}_V)$ and any $t \in H^0(V, \mathcal{T}_V)$:
$$
\left\langle s,  \beta_{t, \varepsilon} t \right\rangle_{\mathcal{O}_V} - \left\langle \beta_{t, \varepsilon}^* s ,  t \right\rangle_{\mathcal{T}_V} = 0,
$$
we deduce that:
\begin{equation}\label{Transconjugue beta local}
    \beta_{t, \varepsilon}^* = \frac{1}{1+t} \sum_{p} \mathrm{d}z_p \otimes \frac{\partial}{\partial z^p}.
\end{equation}

By using the fact that $\omega_{t, \varepsilon}$ is K\"ahler, we have:
\begin{equation}
    \overline{\partial} \circ \beta_{t, \varepsilon} + \beta_{t, \varepsilon} \circ D'' = 0, \qquad
    \partial \circ \beta_{t, \varepsilon} + \beta_{t, \varepsilon} \circ D' = 0.
\end{equation} 

Moreover, we have thanks to \eqref{Beta local} and \eqref{Transconjugue beta local}:
\begin{equation}
    \textbf{i} (1+t)^2 \beta_{t, \varepsilon} \wedge \beta_{t, \varepsilon}^* \wedge \omega_{t, \varepsilon}^{n-1} = -\omega_{t, \varepsilon}^{n},
\end{equation}
\begin{equation}
    \textbf{i} (1+t)^2 \beta_{t, \varepsilon}^* \wedge \beta_{t, \varepsilon} \wedge \omega_{t, \varepsilon}^{n-1} = \frac{1}{n}\omega_{t, \varepsilon}^{n} \otimes \id_{\mathcal{T}_{V}}.
\end{equation}

By replacing $\beta_{t, \varepsilon}$ by $(1+t)\sqrt{\xi}\beta_{t, \varepsilon}$ where $\xi$ is a positive real number (it does not change the isomorphism class of the extension), we get:
\begin{equation}
    \textbf{i}\Theta(\mathcal{V}_t, \nabla_{t, \varepsilon}) \wedge \omega_{t, \varepsilon}^{n-1} = 
    \begin{bmatrix}
        \xi \omega_{t, \varepsilon}^n & 0 \\
        0 & \textbf{i}\Theta(\mathcal{T}_V, h_{t, \varepsilon}) \wedge \omega_{t, \varepsilon}^{n-1} - \frac{\xi}{n} \omega_{t, \varepsilon}^{n} \otimes \id_{\mathcal{T}_V}
    \end{bmatrix}.
\end{equation}
Hence, by choosing $\xi = \frac{1}{n+1}$ (so that $\xi = \frac{1 - \xi}{n}$), we get thanks to \eqref{Perturbed Ric}:
\begin{equation}
    \textbf{i}\Theta(\mathcal{V}_t, \nabla_{t, \varepsilon}) \wedge \omega_{t, \varepsilon}^{n-1} = \frac{1}{n+1} \omega_{t, \varepsilon}^n I_{n+1} + A_{t, \varepsilon},
\end{equation}
where:
$$ 
A_{t, \varepsilon} :=
\begin{bmatrix}
    0 & 0 \\
    0 & \sharp \ddc (\psi_\varepsilon - \varphi_{t, \varepsilon}) - t\sharp \widehat{\omega} - \sum_i a_i \sharp \theta_{i, \varepsilon}
\end{bmatrix}
\frac{1}{n}\omega_{t, \varepsilon}^n.
$$

We can endow $\mathcal{V}_t(E)$ with the metric $h_{\mathcal{V}_t(E), \varepsilon} := h_{\mathcal{V}_t, \varepsilon} \otimes h_E$ where $h_E$ is a metric on $\mathcal{O}(E)$. Then, the curvature tensor associated to this metric is:
\begin{equation}
    \textbf{i}\Theta(\mathcal{V}_t(E)) \wedge \omega_{t, \varepsilon}^{n-1} = \frac{1}{n+1} \omega_{t, \varepsilon}^n I_{n+1} + A_{t, \varepsilon} + \textbf{i}\Theta(E, h_E) \wedge \omega_{t, \varepsilon}^{n-1} I_{n+1}.
\end{equation}

\subsubsection{Semistability of the extension}\label{Semistability of the extension}

If $\mathcal{F}$ is a saturated suborbisheaf of $\widehat{\mathcal{E}}$, then there is an inclusion of orbisheaves $\mathcal{F} \hookrightarrow \mathcal{V}_t(E)$. Since $\mathcal{F}$ could be a subsheaf that is not saturated in $\mathcal{V}_t(E)$, we consider $\mathcal{F}_t$ its saturation. Hence, there exists an effective divisor $D_t$ such that $\det \mathcal{F} = \det \mathcal{F}_t \otimes \mathcal{O}(-D_t)$, and there exists an orbisubbundle $F_t$ of $\mathcal{V}_t(E)$ such that $\mathcal{F}_t$ is the sheaf of sections of this orbibundle $\mathcal{O}(F_t)$ on a big open subset $W_t$. The same arguments as those developed to prove Lemma~\ref{Kobayashi orbifolde slope} show that we still have:
$$
\mu_{\omega_{t, \varepsilon}}(\mathcal{F}_t) = \frac{1}{r} \int_{W_t} c_1^\orb(\mathcal{F}_t, h_{\mathcal{V}_t(E), \varepsilon}) \wedge \omega_{t, \varepsilon}^{n-1}.
$$

Hence, one can observe that:
$$
\lim_{t, \varepsilon \rightarrow 0} \int_{W_t} \tr_{\End} \left( \mathrm{pr}_{F_t} {A_{t, \varepsilon}}_{|F_t} \right) = 0.
$$
Indeed, if $A$ is a hermitian semi-positive endomorphism, the endomorphism
$
\begin{bmatrix}
    0 & 0 \\
    0 & A
\end{bmatrix}
$
remains hermitian and semi-positive, so that the different bounds given in Section~\ref{Analysis of different terms} still guarantee that:
\begin{itemize}
    \item $\displaystyle \int_{W_t} \tr_\End \left( \mathrm{pr}_{F_t}
    \begin{bmatrix}
    0 & 0 \\
    0 & \sharp \ddc (\psi_\varepsilon - \varphi_{t, \varepsilon})
\end{bmatrix}_{|F_t}
\right) \omega_{t, \varepsilon}^n \underset{t, \varepsilon \to 0}{\longrightarrow} 0$,
    \item $\displaystyle \int_{W_t} \tr_\End \left( \mathrm{pr}_{F_t}
    \begin{bmatrix}
    0 & 0 \\
    0 & \sum_i a_i \sharp \theta_{i, \varepsilon}
\end{bmatrix}_{|F_t}
\right) \omega_{t, \varepsilon}^n \underset{t, \varepsilon \to 0}{\longrightarrow} 0$,
    \item $\displaystyle \int_{W_t} \tr_\End \left( \mathrm{pr}_{F_t}
    \begin{bmatrix}
    0 & 0 \\
    0 & t \sharp \widehat{\omega}
\end{bmatrix}_{|F_t}
\right) \omega_{t, \varepsilon}^n \underset{t, \varepsilon \to 0}{\longrightarrow} 0$.
\end{itemize}

Moreover, we have:
$$
\int_{\widehat{X}} c_1(E)^\orb \wedge \omega_{t, \varepsilon}^{n-1} = \left[ E \right] \cdot \left[ \pi^* \omega + t\widehat{\omega} \right] \underset{t, \varepsilon \to 0}{\longrightarrow} 0.
$$

We get that:
\begin{align*}
    \mu_{\omega_{t, \varepsilon}}(\mathcal{F}) & = \mu_{\omega_{t, \varepsilon}}(\mathcal{F}_t) - \frac{1}{r}\int_{\widehat{X}} c_1(D_t) \wedge \omega_{t, \varepsilon}^{n-1} \\
    & \leq \frac{1}{r} \int_{W_t} \tr_\End( \mathrm{pr}_{F_t}(\textbf{i}\Theta(\mathcal{V}_t(E), h_{\mathcal{V}_t(E), \varepsilon})_{|F_t}) \wedge \omega_{t, \varepsilon}^{n-1} \\
    & = \frac{1}{n+1} \int_{\widehat{X}}\omega_{t, \varepsilon}^n + \frac{1}{r} \int_{\widehat{X}} \tr_{\End} \left( \mathrm{pr}_{F_t} {A_{t, \varepsilon}}_{|F_t} \right) + \int_{\widehat{X}} c_1(E)^\orb \wedge \omega_{t, \varepsilon}^{n-1} \\
    & \underset{t, \varepsilon \longrightarrow 0}{\longrightarrow} \frac{1}{n+1} \left[ \pi^* \omega \right]^n = \mu_{\pi^*c_1(X, \Delta)}(\widehat{\mathcal{E}}).
\end{align*}

We deduce that $\widehat{\mathcal{E}}$ is semistable with respect to $\pi^*c_1(X, \Delta)$.

\section{Proof of Theorem A}

\subsection{Reduction to a log resolution}

In order to study the semistability of adapted sheaves $\mathcal{T}_{X, \Delta, f}$ and $\mathcal{E}_{X, \Delta, f}$, we reduce the problem to the log resolution introduced in previous sections. Let us first recall the definition of strictly adapted morphisms and adapted sheaves (see \cite[Sect.~2.1.1 and Sect.~3.2.1]{Dai25} for details).

\begin{defi}
    Let $X, Y$ be normal complex spaces and $\Delta := \sum_i \left( 1 - \frac{1}{m_i} \right) \Delta_i $ be a $\Q$-Weil divisor on $X$ with standard coefficients. A morphism $f : Y \to X$ is \emph{strictly $\Delta$-adapted} if it is surjective, finite, Galois and for any $i$, we have $f^*\Delta_i = m_i \Delta'_i$ where $\Delta'_i$ is a reduced divisor on $Y$.
\end{defi}

\begin{defi}
    Let $(X, \Delta)$ be a klt pair. Then, thanks to \cite[Lemma.~14]{CGG24}, there exists $X^\orb$ an open subset whose complement has codimension at least 3 such that $(X^\orb, \Delta_{|X^\orb})$ can be endowed with a smooth orbi-\'etale orbistructure. Let $f : Y \to X$ a strictly $\Delta$-adapted morphism. By considering normalizations of the fiber products $Y \underset{X}{\times} V_\alpha$, we can construct an orbi-\'etale orbistructure above a big open subset $Y^\circ \subset Y$. If $\mathcal{F}$ is an orbisheaf on $X^\orb$, one can define the pullback orbisheaf $f^* \mathcal{F}$. We define the \emph{adapted sheaf associated to $\mathcal{F}$}, denoted $\mathcal{F}_{X, \Delta, f}$ as the sheaf of invariant sections of $f^*\mathcal{F}$. It is a reflexive sheaf on $Y$.\\
\end{defi}

First, we consider a Kawamata covering $f_1 : Y_1 \to \widehat{X}$ that branches exactly at order $m_i$ over $\widehat{\Delta}_i$ thanks to \cite[Prop.~4.1.12]{Laz04}. The variety $Y_1$ is smooth. Let $(D_i)_i$ be the family of irreducible components of the branching divisor of $f_1$ that are not contained in $\widehat{\Delta}$. We have $f_1^*D_i = \sum_j r_{ij} D'_{ij}$. We set $N_i = \mathrm{lcm}\,(r_{ij})$, so that $f_1$ is a covering branched at most along $\widehat{\Delta} + \sum_i \left( 1 - \frac{1}{N_i} \right) D_i$ in the sense of \cite[Def.~22]{CGG24}. Then, we can consider its Galois closure thanks to \cite[Cor.~27]{CGG24}, let us denote it by $\widehat{f} : \widehat{Y} \to \widehat{X}$. It is a strictly $\widehat{\Delta}$-adapted morphism. Moreover, the morphism $\widehat{Y} \to Y_1$ is quasi-\'etale, hence \'etale because of smoothness of $Y_1$ and purity of the branch locus. Then, $\widehat{Y}$ is smooth.\\

Now, we consider the adapted sheaves associated to the orbibundles previously constructed on $(\widehat{X}, \widehat{\Delta})$: $\mathcal{T}_{\widehat{X}, \widehat{\Delta}, \widehat{f}}, \Omega^1_{\widehat{X}, \widehat{\Delta}, \widehat{f}}, \widehat{\mathcal{E}}_{\widehat{X}, \widehat{\Delta}, \widehat{f}}, \mathcal{V}_t(E)_{\widehat{X}, \widehat{\Delta}, \widehat{f}}$.
These reflexive sheaves are actually locally free, since the local uniformizations above $\widehat{Y}$ are \'etale\footnote{We have even a better description: if $\eta_\alpha : V_\alpha \to U_\alpha$ is a local uniformization on $\widehat{X}$, and we consider $\widehat{Y}_\alpha$ a connected component of the preimage of $U_\alpha$ by $\widehat{f}$, then we have a factorization through $\eta_\alpha$: there exists $\widehat{f}_\alpha : \widehat{Y}_\alpha \to V_\alpha$ such that $f = \eta_\alpha \circ \widehat{f}_\alpha$. This is a consequence of the fact that the local uniformization above $\widehat{Y}$ induced by $\eta_\alpha$ is \'etale, hence it is a local isomorphism.}. Let us denote by:
$$
\xymatrix{
	\widehat{Y} \ar[r]^\nu \ar@/^2pc/[rr]^{\pi \circ \widehat{f}} & Y \ar[r]^f & X
}
$$

the Stein factorization of $\pi \circ \widehat{f}$. Then, we observe that:
\begin{itemize}
	\item the morphism $\nu$ is birational\footnote{It is a consequence of the fact that $\pi \circ \widehat{f}$ is generically finite. Moreover, the finiteness of $f$ ensures that $\nu(E')$ is analytic, with irreducible components of codimension at least 2.}, we denote $E'$ its exceptional locus.
	\item the morphism $f$ is Galois\footnote{Indeed, it is Galois away from the center of $\nu$, and the uniqueness of completion of finite morphisms ensures that we can extend the action through $Y$.} and strictly $\Delta$-adapted\footnote{It suffices to study the local forms of the morphisms around general points of $\Delta$.}.
\end{itemize}

We have the following lemma that is the sought reduction to a resolution.

\begin{lemm}\label{Reduction of semistability to the resolution}
    With notation as above:\\
    $(i)$ if $\mathcal{T}_{\widehat{X}, \widehat{\Delta}, \widehat{f}}$ is $\widehat{f}^*\pi^*c_1(X, \Delta)$-semistable, then $\mathcal{T}_{X, \Delta, f}$ is $f^*c_1(X, \Delta)$-semistable,\\
    $(ii)$ if $\widehat{\mathcal{E}}_{\widehat{X}, \widehat{\Delta}, \widehat{f}}$ is $\widehat{f}^*\pi^*c_1(X, \Delta)$-semistable, then $\mathcal{E}_{X, \Delta, f}$ is $f^*c_1(X, \Delta)$-semistable.
\end{lemm}

\begin{proof}

We first prove $(i)$. By pulling back the orbifold tangent bundle above the orbifold locus of $(X, \Delta)$ through $\pi \circ \widehat{f} = f \circ \nu $, we have the following isomorphisms:
$$
(\nu^{*}\mathcal{T}_{X, \Delta, f})_{|\widehat{Y} \backslash E'} \simeq (\mathcal{T}_{\widehat{X}, \widehat{\Delta}, \widehat{f}})_{|\widehat{Y} \backslash E'}.
$$

We consider $\mathcal{F} \subseteq \mathcal{T}_{X, \Delta, f}$ a proper subsheaf. Let us denote $\widehat{\mathcal{F}}$ the saturation of $(\nu^*\mathcal{F})_{|\widehat{Y} \backslash E'}$ in $\mathcal{T}_{\widehat{X}, \widehat{\Delta}, \widehat{f}}$. Thanks to the projection formula, we have the following equalities of slopes:
$$
\mu_{\nu^*f^*c_1(X, \Delta)}(\widehat{\mathcal{F}}) = \mu_{f^*c_1(X, \Delta)}(\mathcal{F}),
$$
$$
\mu_{\nu^*f^*c_1(X, \Delta)}(\mathcal{T}_{\widehat{X}, \widehat{\Delta}, \widehat{f}}) = \mu_{f^*c_1(X, \Delta)}(\mathcal{T}_{X, \Delta, f}).
$$
From these equalities we deduce the semistability of $\mathcal{T}_{X, \Delta, f}$.

We now prove $(ii)$. We use the same arguments as in $(i)$. In order to do this, it suffices to observe that the projection formula applied to $\widehat{\mathcal{E}}_{\widehat{X}, \widehat{\Delta}, \widehat{f}}$ computes the slope of $\mathcal{E}_{X, \Delta, f}$. We observe that we have an equality of classes in $H^1(\widehat{Y} \backslash E', \Omega^{[1]}_{\widehat{X}, \widehat{\Delta}, \widehat{f}})$
$$
\nu^*f^*(-(K_X + \Delta)) = \widehat{f}^*\pi^*(-(K_X + \Delta)),
$$
we deduce:
$$
\left(\nu^{*}\mathcal{E}_{X, \Delta, f}\right)_{|\widehat{Y} \backslash E'} \simeq (\widehat{\mathcal{E}}_{\widehat{X}, \widehat{\Delta}, \widehat{f}})_{|\widehat{Y} \backslash E'}.
$$
This enables us to compute slopes by projection formula:
$$
\mu_{\nu^*f^*c_1(X, \Delta)}(\widehat{\mathcal{E}}_{\widehat{X}, \widehat{\Delta}, \widehat{f}}) = \mu_{f^*c_1(X, \Delta)}(\mathcal{E}_{X, \Delta, f}).
$$
We conclude the semistability of $\mathcal{E}_{X, \Delta, f}$ similarly to $(i)$.
\end{proof}

The following result shows that the semistability property of adapted sheaves on one strictly $\Delta$-adapted transmits to other such morphisms.

\begin{lemm}\textup{\cite[Prop.~3.11]{Dai25}}
	Let $f : Y \to X$ and $f' : Y' \to X$ two strictly $\Delta$-adapted morphisms, and let $\mathcal{F}$ be an orbisheaf defined on $(X, \Delta)_\orb$. Then, $\mathcal{F}_{X, \Delta, f}$ is $f^*c_1(X, \Delta)$-semistable if and only if $\mathcal{F}_{X, \Delta, f'}$ is $f'^*c_1(X, \Delta)$-semistable. 
\end{lemm}

One can replace "semistable" by "polystable" in the previous statement, and the proof goes similarly, using \cite[Lemma.~3.2.3]{HL10}. Hence, it is sufficient to show that $\mathcal{T}_{\widehat{X}, \widehat{\Delta}, \widehat{f}}$ and $\widehat{\mathcal{E}}_{\widehat{X}, \widehat{\Delta}, \widehat{f}}$ are semistable (or polystable) with respect to $\widehat{f}^*\pi^*c_1(X, \Delta)$.

\subsection{Polystability of the adapted tangent sheaf}

First, we observe that the $(1,1)$-form $\widehat{f}^*\omega_{t, \varepsilon}$ induces a hermitian metric $\widehat{f}^*h_{t, \varepsilon}$ on $\mathcal{T}_{\widehat{X}, \widehat{\Delta}, \widehat{f}}$, 
and a metric $\widehat{f}^*h_{\mathcal{V}_t(E), \varepsilon}$ on $\mathcal{V}_t(E)_{\widehat{X}, \widehat{\Delta}, \widehat{f}}$. In order to keep notation light, we denote in the following:
\begin{itemize}
    \item $\Omega_{t,\varepsilon} := \widehat{f}^*\omega_{t, \varepsilon}$,
    \item $H_{t, \varepsilon} := \widehat{f}^*h_{t, \varepsilon}$ the metric on $\mathcal{T}_{\widehat{X}, \widehat{\Delta}, \widehat{f}}$,
    \item $H^{\mathcal{V}}_{t, \varepsilon} := \widehat{f}^*h_{\mathcal{V}_t(E), \varepsilon}$ the metric on $\mathcal{V}_t(E)_{\widehat{X}, \widehat{\Delta}, \widehat{f}}$.
\end{itemize}
The Ricci curvature of $\Omega_{t, \varepsilon}$ can be written, thanks to \eqref{Perturbed Ric}:
\begin{equation}\label{Ricci perturbed ramified}
    \Ric \Omega_{t, \varepsilon} = \Omega_{t, \varepsilon} + \ddc \widehat{f}^*(\psi_\varepsilon - \varphi_{t, \varepsilon}) - t\widehat{f}^*\widehat{\omega} - \sum_i a_i \widehat{f}^*\theta_{i, \varepsilon} - \sum_j (N_j - 1)\left[ D_j \right],
\end{equation}
where $D_j$ are the irreducible components of the additional ramification, and $N_j$ is the ramification index of $\widehat{f}$ along $D_j$. In the following, we will denote $D := \sum_j D_j$ the reduced divisor of the additional ramification.

\subsubsection{Semistability of $\mathcal{T}_{\widehat{X}, \widehat{\Delta}, \widehat{f}}$}\label{Semistability of the adapted tangent sheaf}

Let us consider $\mathcal{F}$ a proper saturated subsheaf of $\mathcal{T}_{\widehat{X}, \widehat{\Delta}, \widehat{f}}$. Hence, there exists a big open subset $W \subseteq \widehat{Y}$ and a subbundle $F \subseteq \mathcal{T}_{\widehat{X}, \widehat{\Delta}, \widehat{f}}$ such that the restriction of $\mathcal{F}$ to $W$ is the sheaf of sections of $F$. We have by same arguments as those given to prove Lemma~\ref{Kobayashi orbifolde slope}:
$$
\mu_{\widehat{f}^*\omega_{t, \varepsilon}}(\mathcal{F}) = \frac{1}{r}\int_{W} c_1(F, H_{t, \varepsilon}) \wedge \Omega_{t, \varepsilon}^{n-1}.
$$

We have the following inequality pointwise:
$$
    c_1(F, H_{t, \varepsilon}) = \tr_\End(\textbf{i}\Theta(F, H_{t, \varepsilon})) \leq \tr_\End(\mathrm{pr}_F \textbf{i}\Theta(\mathcal{T}_{\widehat{X}, \widehat{\Delta}, \widehat{f}}, H_{t, \varepsilon})_{|F}).
$$
We consider a family of cutoff functions $(\chi_\eta)$ defined on $\widehat{X}$ with compact support included in $\widehat{X} \backslash \widehat{f}(|D|)$, that converges pointwise to 1 when $\eta$ goes to zero. Hence, we have:
\begin{align*}
\mu_{\widehat{f}^*\omega_{t, \varepsilon}}(\mathcal{F}) \leq \frac{1}{r}\int_{W} & \tr_\End(\mathrm{pr}_F \textbf{i}\Theta (\mathcal{T}_{\widehat{X}, \widehat{\Delta}, \widehat{f}}, H_{t, \varepsilon})_{|F}) \wedge \Omega_{t, \varepsilon}^{n-1} \\
& = \frac{1}{r} \lim_{\eta \to 0} \int_{W} \widehat{f}^*\chi_\eta \tr_\End(\mathrm{pr}_F \textbf{i}\Theta(\mathcal{T}_{\widehat{X}, \widehat{\Delta}, \widehat{f}}, H_{t, \varepsilon})_{|F}) \wedge \Omega_{t, \varepsilon}^{n-1} \\
& = \frac{1}{nr} \lim_{\eta \to 0} \int_{W} \widehat{f}^*\chi_\eta \tr_\End(\mathrm{pr}_F \sharp \Ric(\Omega_{t, \varepsilon})_{|F}) \Omega_{t, \varepsilon}^{n}
\end{align*}
and we have by using \eqref{Ricci perturbed ramified}:
\begin{align*}
\frac{1}{nr}\int_{W} \widehat{f}^*\chi_\eta \tr_\End(\mathrm{pr}_F & \sharp \Ric(\Omega_{t, \varepsilon})_{|F}) \Omega_{t, \varepsilon}^{n} \\
& = \frac{1}{nr} \int_{W} \widehat{f}^*\chi_\eta \tr_\End(\mathrm{pr}_F {\sharp \Omega_{t, \varepsilon}}_{|F}) \Omega_{t, \varepsilon}^{n} \\
& \qquad + \frac{1}{nr} \int_{W} \widehat{f}^*\chi_\eta \tr_\End(\mathrm{pr}_F \sharp \ddc \widehat{f}^*(\psi_\varepsilon - \varphi_{t, \varepsilon})_{|F}) \Omega_{t, \varepsilon}^{n} \\
& \qquad - \frac{t}{nr} \int_{W} \widehat{f}^*\chi_\eta \tr_\End(\mathrm{pr}_F \sharp \widehat{f}^*\widehat{\omega}_{|F} ) \Omega_{t, \varepsilon}^{n} \\
& \qquad - \frac{1}{nr} \int_{W} \widehat{f}^*\chi_\eta \sum_i a_i \tr_\End(\mathrm{pr}_F {\sharp \widehat{f}^*\theta_{i, \varepsilon}}_{|F}) \Omega_{t, \varepsilon}^{n}. \\
\end{align*}

The first integral converges to:
$$
\frac{1}{nr}\int_{W} \widehat{f}^*\chi_\eta \tr_\End(\mathrm{pr}_F {\sharp \Omega_{t, \varepsilon}}_{|F}) \Omega_{t, \varepsilon}^{n}  \underset{\eta \to 0}{\longrightarrow} \frac{1}{n} \widehat{f}^* \left[ \omega_{t, \varepsilon} \right]^n.
$$
The limit is a cohomological quantity, and, when $t, \varepsilon$ go to zero, it converges to: 
$$
\frac{1}{n}\widehat{f}^* \left[ \omega_{t, \varepsilon} \right]^n \underset{t, \varepsilon \to 0}{\longrightarrow} \frac{1}{n}\widehat{f}^* \left[ \pi^*\omega \right]^n = \mu_{\widehat{f}^*\pi^*\omega}(\mathcal{T}_{\widehat{X}, \widehat{\Delta}, \widehat{f}}).
$$

Let us analyze the second integral:
$$
\int_{\widehat{Y}} \widehat{f}^*\chi_\eta \tr_\End(\mathrm{pr}_F \sharp \ddc \widehat{f}^*(\psi_\varepsilon - \varphi_{t, \varepsilon})_{|F}) \Omega_{t, \varepsilon}^n.
$$
For any $t, \varepsilon$, there is a domination of the integrand by a $(n,n)$-form $\alpha_{t, \varepsilon}$ with measurable positive coefficients with respect to $\Omega_{t, \varepsilon}$ by using same arguments as those developed for bounding the term \textbf{(III)} in Section~\ref{Analysis of different terms}. By Lebesgue's dominated convergence theorem, this integral converges when $\eta$ goes to zero, to the integral:
$$
\int_{\widehat{Y} \backslash |D|} \tr_\End(\mathrm{pr}_F \sharp \ddc \widehat{f}^*(\psi_\varepsilon - \varphi_{t, \varepsilon})_{|F}) \Omega_{t, \varepsilon}^n,
$$
and the domination indicates that this converges to zero when $\varepsilon, t$ go to zero. Hence, we sill have the convergence of this term to zero. The same argument holds to show the convergence:
$$
 \lim_{t \to 0} \lim_{\varepsilon \to 0} \lim_{\eta \to 0} \int_{W} \widehat{f}^*\chi_\eta \sum_i a_i \tr_\End(\mathrm{pr}_F {\sharp \widehat{f}^*\theta_{i, \varepsilon}}_{|F}) \Omega_{t, \varepsilon}^{n} = 0.
$$

To conclude, the positivity of $\widehat{\omega}$ shows that:
$$
\int_{W} \widehat{f}^*\chi_\eta \tr_\End(\mathrm{pr}_F \sharp \widehat{f}^*\widehat{\omega}_{|F} ) \Omega_{t, \varepsilon}^{n} \geq 0.
$$

Finally, we have by taking limit when $\varepsilon, t$ go to zero:
$$
\mu_{\widehat{f}^*\pi^*\omega}(\mathcal{F}) \leq \mu_{\widehat{f}^*\pi^*\omega}(\mathcal{T}_{\widehat{X}, \widehat{\Delta}, \widehat{f}}),
$$
and we conclude that $\mathcal{T}_{\widehat{X}, \widehat{\Delta}, \widehat{f}}$ is semistable.

\subsubsection{Polystability of $\mathcal{T}_{X, \Delta, f}$}\label{Polystability of the adapted tangent sheaf}

To show polystability, we adapt the arguments from the proof of \cite[Thm.~6]{DGP24}. Let $\mathcal{F}$ be a saturated coherent subsheaf of $\mathcal{T}_{\widehat{X}, \widehat{\Delta}, \widehat{f}}$ such that
$$
\mu_{\widehat{f}^*\pi^* \omega}(\mathcal{F}) = \mu_{\widehat{f}^*\pi^* \omega}(\mathcal{T}_{\widehat{X}, \widehat{\Delta}, \widehat{f}}).
$$
Thanks to Remark~\ref{convergence local smooth orbifold of potentials}, we observe that $\Omega_{t, \varepsilon}$ converge $\mathcal{C}^\infty$ locally in $\widehat{Y} \backslash E'$ to $\widehat{f}^*\pi^*\omega$ so that the second fundamental forms $\beta_{t, \varepsilon}$ given by the inclusion of vector bundles $F \hookrightarrow \mathcal{T}_{\widehat{X}, \widehat{\Delta}, \widehat{f}}$ on $W^\circ := W \cap \widehat{Y} \backslash E'$ induced by $H_{t, \varepsilon}$ converge to $\beta_\infty$, the second fundamental form induced by $\widehat{f}^*\pi^* \omega$ away from $E'$. By Fatou lemma, we have:
$$
0 \leq -\int_{W^\circ} \tr_\End(\beta_\infty \wedge \beta_\infty^*) \wedge \widehat{f}^*\pi^*\omega \leq \liminf_{t, \varepsilon \to 0} -\int_{W^\circ} \tr_\End(\beta_{t, \varepsilon} \wedge \beta_{t, \varepsilon}^*) \wedge \Omega_{t, \varepsilon}^{n-1} = 0.
$$

Hence, we have a holomorphic splitting: $\mathcal{T}_{\widehat{X}, \widehat{\Delta}, \widehat{f} |W^\circ} = \mathcal{F}_{| W^\circ} \oplus (\mathcal{F}_{|W^\circ})^\perp$.
By iterating this process and starting with $\mathcal{F}$ which has same slope as $\mathcal{T}_{\widehat{X}, \widehat{\Delta}, \widehat{f}}$ and has minimal rank, we have a decomposition:
$$
\mathcal{T}_{\widehat{X}, \widehat{\Delta}, \widehat{f} |W^\circ} = \bigoplus_i \mathcal{F}_{i |W^\circ}.
$$

Now, we denote by $j : V := \nu(W^\circ) \hookrightarrow Y$ the inclusion morphism of $V$ which is a Zariski open subset whose complement has codimension at least 2. We have:
$$
j_* \nu^{-1}_* \mathcal{T}_{\widehat{X}, \widehat{\Delta}, \widehat{f}} = \bigoplus_i j_* \nu^{-1}_*\mathcal{F}_{i |W^\circ}.
$$
We have $j_* \nu^{-1}_* \mathcal{T}_{\widehat{X}, \widehat{\Delta}, \widehat{f}} = \mathcal{T}_{X, \Delta, f}$ since they are both reflexive and they coincide on $V$. Moreover, the sheaves $\mathcal{S}_i := j_* \nu^{-1}_*\mathcal{F}_{i |W^\circ}$ are reflexive, and they are stable\footnote{The computations from Lemma~\ref{Reduction of semistability to the resolution} show that we can replace semistable by stable.} with respect to $f^*c_1(X, \Delta)$. By projection formula, they all have same slope equal to $\mu_{f^*\omega}(\mathcal{T}_{X, \Delta, f})$.\\

\subsection{Polystability of the adapted canonical extension}

\subsubsection{Semistability of $\mathcal{E}_{\widehat{X}, \widehat{\Delta}, \widehat{f}}$}
Let us consider $\mathcal{F}$ a proper saturated subsheaf of $\widehat{\mathcal{E}}_{\widehat{X}, \widehat{\Delta}, \widehat{f}}$. Then, there is an inclusion of sheaves $\mathcal{F} \hookrightarrow \mathcal{V}_t(E)_{ \widehat{X}, \widehat{\Delta}, \widehat{f}}$ for $t$ in a neighborhood of $0$ by same arguments as Section~\ref{Inclusion of suborbisheaf}. Now, since $\mathcal{F}$ is not necessarily saturated in $\mathcal{V}_t(E)_{ \widehat{X}, \widehat{\Delta}, \widehat{f}}$, we consider its saturation $\mathcal{F}_t$. Hence, there exists an effective divisor $D_t$ such that $\det \mathcal{F} = \det \mathcal{F}_t \otimes \mathcal{O}(-D_t)$ and an open subset $W_t \subseteq \widehat{Y}$ and a subbundle $F_t \subseteq \mathcal{V}_t(E)_{\widehat{X}, \widehat{\Delta}, \widehat{f}}$ such that $\mathcal{F}_t$ is the sheaf of sections of $F_t$ on $W_t$.
First, we have:
$$
    \mu_{\widehat{f}^*\omega_{t, \varepsilon}}(\mathcal{F}) = \mu_{\widehat{f}^*\omega_{t, \varepsilon}}(\mathcal{F}_t) - \frac{1}{r} \left[ D_t \right] \cdot \left[ \widehat{f}^*\omega_{t, \varepsilon}  \right]^{n-1} \leq \mu_{\widehat{f}^*\omega_{t, \varepsilon}}(\mathcal{F}_t).
$$

We have by same arguments as in the previous section:
$$
    \mu_{\widehat{f}^*\omega_{t, \varepsilon}}(\mathcal{F}_t) = \frac{1}{r} \int_{W_t} c_1(F_t, H^{\mathcal{V}}_{t, \varepsilon}) \wedge \Omega_{t, \varepsilon}^{n-1}.
$$

We have the following inequality pointwise:
\begin{align*}
    c_1(F_t, H^{\mathcal{V}}_{t, \varepsilon}) & \leq \tr_\End(\mathrm{pr}_{F_t} \textbf{i}\Theta(\mathcal{V}_t(E)_{\widehat{X}, \widehat{\Delta}, \widehat{f}}, H^{\mathcal{V}}_{t, \varepsilon})_{|F_t}).
\end{align*}
Away from $D$, the previous curvature tensor of the deformation of the extension can be written in the following way:
\begin{align*}
    \textbf{i}\Theta(\mathcal{V}_t(E)_{\widehat{X}, \widehat{\Delta}, \widehat{f}}, H^{\mathcal{V}}_{t, \varepsilon}) \wedge \Omega_{t, \varepsilon}^{n-1} & = \frac{1}{n+1} \Omega_{t, \varepsilon}^n I_{n+1} + A'_{t, \varepsilon} \\
    & \qquad + \widehat{f}^*\textbf{i}\Theta(E, h_E) \wedge \Omega_{t, \varepsilon}^{n-1} I_{n+1},
\end{align*}
where 
$$
    A'_{t, \varepsilon} :=
\begin{bmatrix}
    0 & 0 \\
    0 & \sharp \ddc \widehat{f}^*(\psi_\varepsilon - \varphi_{t, \varepsilon}) - t\sharp \widehat{f}^*\widehat{\omega} - \sum_i a_i \sharp \widehat{f}^*\theta_{i, \varepsilon}
\end{bmatrix}
\frac{1}{n}\Omega_{t, \varepsilon}^n.
$$
The same arguments as those developed at the end of Section~\ref{Semistability of the extension} conjugated with the bounds given in Section~\ref{Semistability of the adapted tangent sheaf} show that:
$$
\int_{W_t} \tr_\End(\mathrm{pr}_{F_t} {A'_{t, \varepsilon}}_{|F_t}) \underset{t, \varepsilon \to 0}{\longrightarrow} 0,
$$
and similarly
$$
\int_{\widehat{Y}} \widehat{f}^*\textbf{i}\Theta(E, h_E) \wedge \Omega_{t, \varepsilon}^{n-1} \underset{t, \varepsilon \to 0}{\longrightarrow} 0.
$$

Hence, we have:
$$
\frac{1}{r}\int_{W_t} \tr_\End(\mathrm{pr}_{F_t}\textbf{i}\Theta(\mathcal{V}_t(E)_{\widehat{X}, \widehat{\Delta}, \widehat{f}}, H^{\mathcal{V}}_{t, \varepsilon})_{|F_t}) \Omega_{t, \varepsilon}^n \underset{t, \varepsilon \to 0}{\longrightarrow}\frac{1}{n+1} \widehat{f}^* \left[ \pi^*\omega \right]^{n} = \mu_{\widehat{f}^*\pi^*\omega}(\widehat{\mathcal{E}}_{\widehat{X}, \widehat{\Delta}, \widehat{f}}).
$$
In other words, by taking the limit in $t, \varepsilon$, we get that the adapted extension $\widehat{\mathcal{E}}_{\widehat{X}, \widehat{\Delta}, \widehat{f}}$ is semistable:
$$
\mu_{\widehat{f}^*\pi^*\omega}(\mathcal{F}) \leq \mu_{\widehat{f}^*\pi^*\omega}(\widehat{\mathcal{E}}_{\widehat{X}, \widehat{\Delta}, \widehat{f}}).
$$

\subsubsection{Polystability of $\mathcal{E}_{X, \Delta, f}$}

To show the polystability of the canonical extension, we use the same arguments as those developed for the polystability of $\mathcal{T}_{X, \Delta, f}$. We consider $\mathcal{F} \subseteq \widehat{\mathcal{E}}_{\widehat{X}, \widehat{\Delta}, \widehat{f}}$ a proper coherent saturated subsheaf that has same slope as the extension. Then, the convergence properties of the second fundamental forms stated in \S\ref{Polystability of the adapted tangent sheaf} show that above a Zariski open subset $W_0^\circ \subseteq \widehat{Y} \backslash E'$, we have a decomposition: $\widehat{\mathcal{E}}_{\widehat{X}, \widehat{\Delta}, \widehat{f} |W_0^\circ} = \mathcal{F}_{|W_0^\circ} \oplus (\mathcal{F}_{|W_0^\circ})^\perp$. By induction, this produces a decomposition:
$$
\widehat{\mathcal{E}}_{\widehat{X}, \widehat{\Delta}, \widehat{f} |W_0^\circ} = \bigoplus_i \mathcal{F}_{i |W_0^\circ},
$$ 

and by pushing forward this decomposition to $V_0 := \nu(W_0^\circ)$, and taking the direct image through the inclusion $j_0 : V_0 \hookrightarrow Y$, we get the sought decomposition.

\printbibliography

\end{document}